\documentclass{amsart}
\usepackage{amsfonts,amssymb,amsmath,amsthm}
\usepackage{url}
\usepackage{enumerate}
\usepackage{fullpage}
\usepackage{bm}
\usepackage{wasysym}
\usepackage[usenames,dvipsnames]{color}
\usepackage{amscd}
\usepackage{graphicx}
\usepackage{subcaption}
\usepackage{tikz}
	\usetikzlibrary{positioning}

\newtheorem{thm}{Theorem}[section]
\newtheorem{lem}[thm]{Lemma}
\newtheorem{prop}[thm]{Proposition}
\newtheorem{cor}[thm]{Corollary}
\theoremstyle{definition}
\newtheorem{defn}[thm]{Definition}
\newtheorem{remk}[thm]{Remark}

\newtheorem{conj}[thm]{Conjecture}

\newtheorem{quest}[thm]{Question}
\numberwithin{equation}{section}

    \newtheoremstyle{TheoremNum}
        {\topsep}{\topsep}              
        {\itshape}                      
        {}                              
        {\bfseries}                     
        {.}                             
        { }                             
        {\thmname{#1}\thmnote{ \bfseries #3}}
    \theoremstyle{TheoremNum}
    \newtheorem{thmn}{Theorem}

\newcommand{\A}{\mathbb{A}}      
\newcommand{\C}{\mathbb{C}}

\newcommand{\R}{\mathbb{R}}
\newcommand{\Z}{\mathbb{Z}}
\newcommand{\Q}{\mathbb{Q}}
\newcommand{\PP}{\mathbb{P}}

\newcommand{\TT}{\mathbb{T}}
\newcommand{\m}{\mathrm{m}}

\newcommand{\J}{\mathcal J}
\newcommand{\K}{\mathcal K}

\newcommand{\Aut}{\operatorname{Aut}}

\newcommand{\Gal}{\operatorname{Gal}}

\newcommand{\PrePer}{\operatorname{PrePer}}

\author{Annie Carter}
\author{Matilde Lal\'in}
\author{Michelle Manes}
\author{Alison Beth Miller} 

\address{Annie Carter: Department of Mathematics, University of California San Diego, 9500 Gilman Drive \# 0112, La Jolla, CA  92093-0112, USA}\email{a4carter@ucsd.edu}

\address{Matilde Lal\'in:  D\'epartement de math\'ematiques et de statistique,
                                    Universit\'e de Montr\'eal.
                                    CP 6128, succ. Centre-ville.
                                     Montreal, QC H3C 3J7, Canada}\email{mlalin@dms.umontreal.ca}

\address{Michelle Manes: Department of Mathematics, University of Hawaii, 2565 McCarthy Mall, Honolulu, HI 96822, USA}\email{mmanes@math.hawaii.edu}

\address{Alison Beth Miller: Mathematical Reviews, 416 Fourth St., Ann Arbor, MI 48103, USA}\email{alimil@umich.edu }

\title{Dynamical Mahler Measure: A survey and some recent results}



\begin{document}

	\begin{abstract}
We study the dynamical Mahler measure of multivariate polynomials and present dynamical analogues of various results from the classical Mahler measure as well as examples of formulas allowing the computation of the dynamical Mahler measure in certain cases. We discuss multivariate analogues of dynamical Kronecker's Lemma and present some improvements on the result for two variables due to  Carter, Lal\'in, Manes, Miller, and Mocz.
	\end{abstract}

	\maketitle
	

\section{Introduction}
The inspiration for our investigation comes from the following result which relates 
the canonical height $\hat{h}_f$ (see Definition~\ref{height-definition}) of a point $\alpha \in \PP^1(\overline{\Q})$ relative to some polynomial $f\in \Q[z]$ with an integral of the minimal polynomial of $\alpha$ relative to an invariant measure defined by $f$.

\begin{thm}[\cite{PST}]
Let $f\in\Q[z]$ be a polynomial, and let
$\J_{f}$ denote the Julia set of $f$. Let $K$ be a number field with $\alpha \in \PP^1(K)$, and let $P \in \Z[x]$ be the minimal polynomial for $\alpha$.  Then:

\begin{equation}\label{eq:dynMah1D}
[\Q(\alpha):\Q]\hat h_f(\alpha) = 
\int_{\J_{f}}\log |P(z)| d\mu_{f}(z).
\end{equation}

\end{thm}

Compare this with a standard formula relating the Mahler measure of the minimal polynomial $P$ and the height of a root of that polynomial:
\begin{equation}\label{eq:mahlerheight}
[\Q(\alpha) : \Q] h(\alpha)  
= \underbrace{\frac{1}{2\pi i}\int_{\TT^1}  \log \left| P\left(z \right)\right| \frac{dz}{z}}_{\m(P)}.
\end{equation}

The tantalizing similarities in these formulas lead naturally to questions about extending classical results of Mahler measure to this new ``dynamical Mahler measure'' relative to a fixed polynomial~$f$.   In~\cite{TwoVarPoly}, the authors define a multivariate  dynamical Mahler measure and prove several preliminary results with this flavor.  
This survey article presents background, motivation, examples, and strengthening of those results, both illustrating and expanding on the work begun in~\cite{TwoVarPoly}.  

 In Section~\ref{sec:basics}, we provide background on Mahler measure, arithmetic dynamics, and equilibrium measures. In Section~\ref{sec:DMMdef}, we define the multivariate dynamical Mahler measure and give examples where it  is possible to compute it exactly. Section~\ref{sec:PrevPaper} provides a summary of results from~\cite{TwoVarPoly}, drawing explicit connections between classical Mahler measure and the dynamical setting. In Section~\ref{sec:DynMahlerConverge}, we prove the existence of dynamical Mahler measure as defined in the previous section; these proofs also appear in~\cite{TwoVarPoly} but are reiterated here (with a bit more detail) to provide a self-contained reference to the subject. Section~\ref{sec:DynKronLem} contains a survey of recent results on properties that are either known or conjectured to be equivalent to a multivariate polynomial having dynamical Mahler measure zero. Section~\ref{sec:e1e2} contains the proof of a new implication of this sort, and Section~\ref{sec:ad} contains strengthening of one of the results from~\cite{TwoVarPoly}. In particular, in the proof of the two-variable Dynamical Kronecker's Lemma, we replace the (rather strong) assumption of Dynamical Lehmer's Conjecture with an assumption about the preperiodic points for the polynomial $f$. Finally, Section~\ref{sec:J=K} investigates that condition on polynomials $f$.

\subsection*{Acknowledgements} We are grateful to Patrick Ingram for suggesting that we study the dynamical Mahler measure of multivariate polynomials. Thanks to Patrick Ingram and to Rob Benedetto for many helpful discussions. We thank the organizers of the BIRS workshop ``Women in Numbers 5'',  Alina Bucur, Wei Ho, and Renate Scheidler, for their leadership and encouragement that extended for the whole duration of this project. This work has been partially supported by 
the Natural Sciences and Engineering Research Council
of Canada (Discovery Grant 355412-2013 to ML), the Fonds de recherche du Qu\'ebec - Nature et technologies (Projets de recherche en \'equipe 256442 and 300951 to ML), the Simons Foundation (grant number 359721 to MM), and the National Science Foundation (grant DMS-1844206 supporting AC).

\section{Basic Notions}
\label{sec:basics}
In this section, we provide preliminary material on both Mahler measure and arithmetic dynamics. We refer the interested reader to~\cite{Bertin-Lalin}  for a more comprehensive  article describing the history and applications of Mahler measure in arithmetic geometry and to~\cite{CurrentTrends}  for background and motivation from the arithmetic dynamics perspective.

\subsection{Mahler Measure}\label{sec:ClassicalMahler}
The (logarithmic) Mahler measure of a non-zero polynomial $P\in \C[x]$, originally defined by Lehmer~\cite{Le}, is a height function given by
\begin{equation}\label{eq:defMmeas}
\m(P) = \m\left(a \prod_j(x-\alpha_j)\right)=\log |a| + \sum_j \log \max \{1,|\alpha_j|\}.
\end{equation}

If $P\in \Z[x]$, the formula above makes it clear that $\m(P) \geq 0$.   In such a case, it is natural to ask which polynomials $P\in \Z[x]$ satisfy $\m(P)=0$. 
A result of Kronecker~\cite{Kronecker} gives the answer.

 \begin{lem}[Kronecker's Lemma] Let $P\in \Z[x]$. Then $\m(P)=0$ if and only if $P$ is monic and can be decomposed as a product of a monomial and cyclotomic polynomials. 
\end{lem}

Lehmer~\cite{Le} computed 
\[ \m(x^{10}+x^9-x^7-x^6-x^5-x^4-x^3+x+1) = \log(1.176280818\dots) = 0.162357612\dots \]
and asked the following:

\begin{quest}[Lehmer's question, 1933]
Is there a constant $C > 0$ such that for every polynomial $P \in \Z[x]$ with $\m(P) >0$, then $\m(P)\geq C$?
\end{quest}
Lehmer's question remains open, and his degree-10 polynomial remains the integer polynomial with the smallest known positive measure.

 Jensen's formula~\cite{Jensen} relates an average of a linear polynomial over the unit circle with the size of its root:
 \begin{equation}\label{eq:Jensen}
\frac{1}{2\pi i}\int_{\TT^1} \log \left|x - \alpha \right| \frac{dx}{x} = \log \max\{ 1, |\alpha| \}.
\end{equation}

Applying Jensen's formula to the definition of Mahler measure in~\eqref{eq:defMmeas}, we find a formula that can be extended naturally to multivariate polynomials  and   rational functions. 
Following Mahler~\cite{Mah}, we have:
\begin{defn}\label{defn:MM}The (logarithmic) Mahler measure of a non-zero rational function $P \in \C(x_1,\dots,x_n)$ is  defined by 
\begin{equation*}
 \m(P):=\frac{1}{(2\pi i)^n}\int_{\mathbb{T}^n}\log|P(x_1,\dots, x_n)|\frac{dx_1}{x_1}\cdots \frac{dx_n}{x_n},
\end{equation*}
where $\mathbb{T}^n=\{(x_1,\dots,x_n)\in \mathbb{C}^n : |x_1|=\cdots=|x_n|=1\}$. 
\end{defn}
The above integral converges, and for  $P\in \Z[x_1,\dots,x_n]$, we still have $\m(P)\geq 0$ (see Proposition~\ref{prop:convergence}). It is natural, then,  to consider whether Kronecker's Lemma has an extension to multivariate polynomials. Recall that a polynomial in $\Z[x_1,\dots,x_n]$ is said to be \textbf{primitive} if the coefficients have no non-trivial factor. We have the following result. 

\begin{thm}\cite[Theorem 3.10]{Everestward} \label{thm:HigherdKronecker}
For any primitive polynomial $P\in \Z[x_1,\dots,x_n]$, we have $\m(P) =0$ if and only if $P$ is the product of a monomial and cyclotomic polynomials evaluated on monomials. 

\end{thm}

A  connection between the single-variable case and the multivariate case is given by a result due to Boyd ~\cite{Boyd-speculations} and Lawton~\cite{Lawton}.  

\begin{thm}\cite[Theorem 2]{Lawton} If $P\in \C(x_1,\dots,x_n)^\times$, then 
\begin{equation}\label{eq:BL}
\lim_{q({\bf k})\rightarrow \infty} \m (P(x,x^{k_2},\dots,x^{k_n}))=\m(P(x_1,\dots,x_n)),
\end{equation}
where 
\[q({\bf k})=\min \left\{ H({\bf s}) : {\bf s}=(s_2,\dots,s_n) \in \Z^{n-1}, {\bf s} \not = (0,\dots,0), \mbox{ and }\sum_{j=2}^n s_j k_j =0\right\}\]
and $H({\bf s})=\max\{|s_j|: 2\leq j \leq n\}$.  
\end{thm}
 Intuitively, the second equation says that the limit is taken while $k_2,\dots,k_n$ go to infinity independently from each other.

Mahler measure often yields special values of interesting number-theoretic functions,  such as 
the Riemann zeta function and $L$-functions associated to arithmetic-geometric objects such as elliptic curves. 
For more on these connections, see~\cite{Bertin-Lalin,BrunaultZudilin}.

\subsection{Arithmetic Dynamics}\label{sec:ArithDynIntro}
A discrete dynamical system is a set $X$ together with a self-map: 
 $f: X \to X$, allowing for iteration. Here we focus on polynomials $f: \C \to \C$. For such an $f$ and for $L \in \Aut(\C)$ (so $L = az + b \in \C[x]$), we write 
\begin{equation}\label{def:conj-notation}
f^n = \underbrace{f\circ f \circ \cdots \circ f}_{\text{$n$-fold composition}},
\quad\text{ and }\quad f^L := L^{-1}\circ f\circ L.
\end{equation}
We will say that $f^L$ and $f$ are affine conjugate over $K$ when $L \in K[x]$. This conjugation is a natural dynamical equivalence relation because it respects iteration: $(f^L)^n = (f^n)^L$.  

A fundamental goal of dynamics is to study the behavior of points of $X$ under iteration. For example, a point $\alpha \in X$ is said to be:
\begin{itemize}
\item {\bf periodic} if $f^n(\alpha) = \alpha$ for some $n>0$,
\item {\bf preperiodic} if $f^n(\alpha) = f^m(\alpha)$ for some $n>m\geq0$, and
\item {\bf wandering} if it is not preperiodic.
\end{itemize}

We write 
\[
\PrePer(f) = \{\alpha \in X : \alpha \text{ is preperiodic under } f\}.
\]

As usual, we say that $\alpha$ is a \textbf{critical} point if $f'(\alpha) = 0$. Critical points play an important role in analyzing the dynamics of the function $f$.

 Questions in arithmetic dynamics are often motivated by an analogy between arithmetic geometry and dynamical systems in which, for example, rational and integral points on varieties correspond to rational and integral points in orbits, and torsion points on abelian varieties correspond to preperiodic points.  It should be no surprise, then, that heights are an essential tool in the study of arithmetic dynamics.
   
 We recall that the classical (logarithmic) height of a rational number $\alpha = \frac{a}{b} \in \Q$, written in lowest terms, is $h(\alpha) = \log\max\{|a|, |b|\}$. This can be extended naturally to a height on algebraic numbers.
 
One way of making such an extension is to consider the na\"ive height. Let $\alpha$ be an algebraic number. We consider its minimal polynomial $P_\alpha(z)=\sum_{j=0}^n a_j z^n$ normalized such that it has integral coefficients  and is primitive. Then 
\[h_{\text{na\"ive}}(\alpha):=\log \max_{j} |a_j|.\]
 
Another possible extension of the classical height is given by the Weil height. For  $\alpha \in K$, with $K$ a number field, the  (absolute logarithmic)  Weil height  is given by 
 \begin{equation}\label{eq:WeilHeight}
 h_{\text{Weil}}(\alpha)=\frac{1}{[K:\Q]}\sum_{\substack{v\in M_K\\v\mid p}} [K_v:\Q_p] \log \max \{||\alpha||_v,1\},
 \end{equation}
 where $M_K$ is an appropriately normalized set of inequivalent absolute values on $K$, so that the product formula is satisfied:
 \[\prod_{\substack{v\in M_K\\v\mid p}} ||x||_v^\frac{[K_v:\Q_p] }{ [K:\Q] }=1.\]
More concretely,  for $K=\Q$ we can take $|\cdot|_\infty$ to be the usual absolute value and $|\cdot|_p$ to be the $p$-adic absolute value, normalized so that $|p|_p=1/p$. Then  for $v \in M_K$ lying over a prime $p$, 
 \[||x||_v=|N_{K_v/\Q_p}(x)|_p^\frac{1}{[K_v:\Q_p]}.\]
 The factor of $[K:\Q]$ in~\eqref{eq:WeilHeight} ensures that $h_{\text{Weil}}(\alpha)$ is well-defined, with the same answer for any field $K$ containing $\alpha$.
 
 While the na\"ive height is very natural to consider, the Weil height is the canonical height for the power map $z \mapsto z^d$. The two heights are commensurate in the sense that  for $\alpha \in \PP^1(\overline{\Q})$, there is a constant $C(d)$ depending only on the degree $d$ of $\alpha$ such that
 \begin{equation}\label{eq:heightequiv}|h_{\text{na\"ive}}(\alpha)-h_{\text{Weil}}(\alpha)|\leq C(d).\end{equation}

 If $f(x) \in K(x)$ is a rational function of degree $d$, then $h\left(f(\alpha)\right)$ should be approximately $d h(\alpha)$. The dynamical canonical height makes this an equality. The definition is reminiscent of the N\'eron--Tate height on an abelian variety, and the proofs of the statements below follow exactly as in this more familiar case.

\begin{defn}
\label{height-definition}
If $f \in \overline\Q(z)$ is a rational map of degree $d$ (i.e.\ the maximum of the degrees of the numerator and denominator is $d$), and $\alpha \in \PP^1(\overline{\Q})$, we then define:
\[
\hat h_f(\alpha) = \lim_{n \to \infty} \frac{h(f^n(\alpha))}{d^n},
\]
where $h$ may be taken as either the na\"ive height $h_{\text{na\"ive}}$ or the Weil height $h_{\text{Weil}}$ by equation \eqref{eq:heightequiv}
\end{defn}

It is known that this limit exists, that 
$\hat h_f(f(\alpha)) = d \hat h_f(\alpha)$, and
that $\hat  h_f(\alpha) = 0$ if and only if $\alpha$ is a preperiodic point for $f$. See Section 3.4 of~\cite{Silverman-arithmetic-dynamical} for details.

\begin{defn}\label{def:Julia}
Let $f\in \C[z]$. The {\bf filled Julia set} of $f$ is
 \[\K_f=\{z\in \C \, :\, f^n(z)\not \rightarrow \infty \mbox{ as } n\rightarrow \infty\}.\]
 The {\bf Julia set} $\J_f$ of $f$ is the boundary of the filled Julia set. That is, $\J_f = \partial\K_f$.
 \end{defn}
 
It follows from these definitions that for a polynomial $f\in \C[z]$, both $\K_f$ and $\J_f$ are compact. 
We denote by $F_\infty$ the complement of $\K_f$ in $\PP^1(\C)$, which is also the attracting basin of $\infty$ for $f$, that is, the set of points in $\PP^1(\C)$ whose orbits go off to $
\infty$.

For example, for $f(z) = z^d$, we see that
$f^n(z) = z^{d^n}$. So
for $d\geq 2$,  we have three cases: 
\begin{itemize}
\item
If $|\alpha| >1$ then $|\alpha^{d^n}|  \to \infty $ with $n$.
\item
If $|\alpha| <1$ then $|\alpha^{d^n}|  \to 0  $ with $n$.
\item
If $|\alpha| =1$ then $|\alpha^{d^n}|  = 1 $ for all $n$.
\end{itemize}
So for  pure power maps, we can understand the Julia sets completely: $\K_f$ is the unit disc, and $\J_f$ is the unit circle. In general, however, these sets are quite complex. (See Figure~\ref{fig:JuliaSets}.)

\begin{figure}[ht!]
\captionsetup[subfigure]{justification=centering}
    \centering
    \begin{subfigure}[t]{0.25\textwidth}
        \centering
        \includegraphics[height=1.2in]{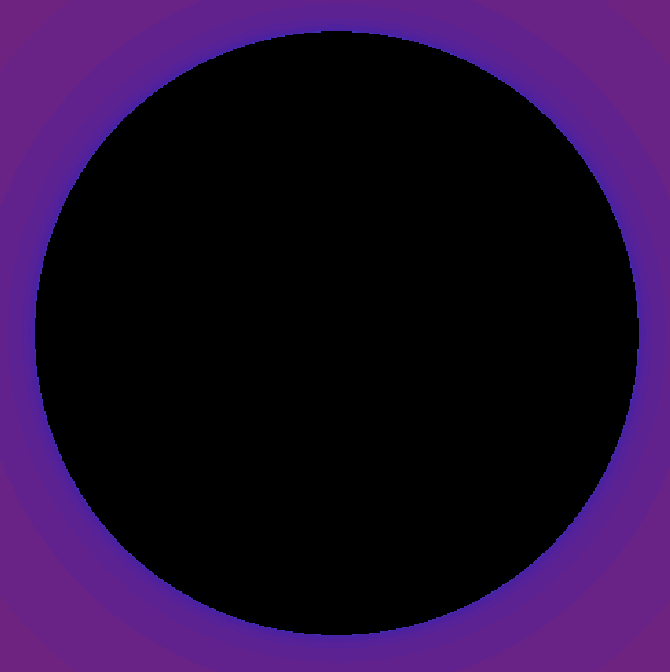}
        \caption{Filled Julia set for \\$f(z) = z^2$}
    \end{subfigure}%
    ~ 
    \begin{subfigure}[t]{0.5\textwidth}
        \centering
        \includegraphics[height=1.2in]{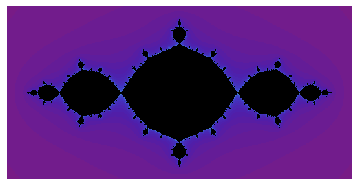}
        \caption{Filled Julia set for \\$f(z) = z^2-1$}
    \end{subfigure}
    ~
        \begin{subfigure}[t]{0.25\textwidth}
        \centering
        \includegraphics[height=1.2in]{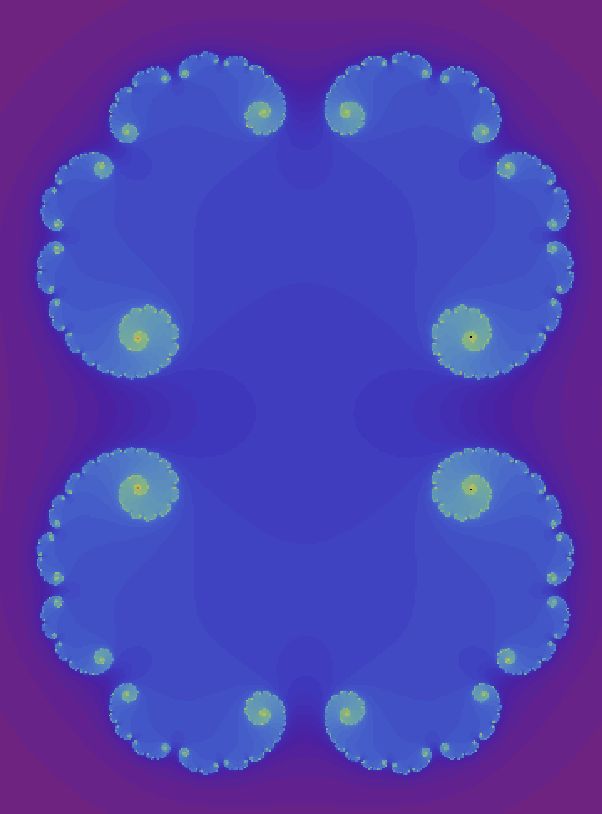}
        \caption{(Filled) Julia set for \\$f(z) = z^2+0.3$}
    \end{subfigure}
\caption{The black area shows the filled Julia set $\K_f$, and its boundary is the Julia set $\J_f$. In the third case, the Julia set has empty interior, so $\K_f = \J_f$.}
    \label{fig:JuliaSets}
\end{figure}

It is clear from Definition~\ref{def:Julia} that $\PrePer(f) \subseteq \K_f$.
When all of the critical points of a polynomial $f$ have unbounded orbits, the Julia set $\J_f$ is totally disconnected, while $\J_f$ is connected if and only if all the critical orbits are bounded~\cite{Fatou, Julia}. A polynomial of degree 2 has a unique critical point, and we have only these two cases (see Figure~\ref{fig:JuliaSets}). This situation is known as the Fatou--Julia dichotomy. However, for higher degree polynomials the situation is complicated by having more than one critical point, and there are cases of polynomials $f$ for which $\J_f$ is disconnected but not totally disconnected (see Figure~\ref{fig:disconnectedJulia}). 

\begin{figure}[h]
\centering
\includegraphics[scale=.5]{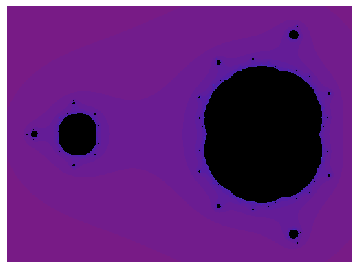}
\caption{The filled Julia set for $f(z) =z^3-z+1$. The Julia set $\J_f$ is disconnected but not totally disconnected.}
\label{fig:disconnectedJulia}
\end{figure}

We saw in equations~\eqref{eq:dynMah1D} and~\eqref{eq:mahlerheight} that  the Julia set $\J_f$ for a  polynomial $f$ will play the role of the unit torus $\TT^1$ when studying dynamical Mahler measure.

\subsection{Equilibrium Measures}

\begin{defn}
Given a compact subset $K \subseteq \C$, an \textbf{equilibrium measure} for $K$ is a Borel probability measure $\mu$ on $K$ which has maximal \textbf{energy}
	\[ I(\mu) := \int_K \int_K \log|z - w|\ d\mu(z)\ d\mu(w)\]
among all Borel probability measures on $K$.
\end{defn}

Every compact set $K \subseteq \mathbb{C}$ has an equilibrium measure~\cite[Theorem 3.3.2]{Ransford}, and if $f$ denotes a polynomial of degree $d \geq 2$, then the equilibrium measure $\mu_f$ on its Julia set $\J_f$ is unique (this is the consequence of a more general result that states that the equilibrium measure of any compact, \emph{non-polar} set is unique~\cite[Theorem 3.7.6]{Ransford}; the Julia set $\J_f$ of any polynomial $f$ is non-polar~\cite[Theorem~6.5.1]{Ransford}), which is to say that there is a non-trivial finite  Borel measure with compact support such that $I(\mu)>-\infty$. 
In fact we can characterize the equilibrium measure $\mu_f$ as follows:

\begin{thm}[{\cite[Theorem 6.5.8]{Ransford}}]
\label{equilibrium-measure-sequence}
Let $w \in \J_f$, and for $n \geq 1$, define the Borel probability measures
	\[ \mu_n := \frac 1{d^n} \sum_{f^n(\zeta) = w} \delta_{\zeta}, \]
where $\delta_{\zeta}$ denotes the unit mass at $\zeta$, and the preimages $\zeta$ of $w$ under $f^n$ are taken with multiplicity. Then $\mu_n \overset{w^\ast}{\to} \mu_f$ (weak$^\ast$-convergence) as $n \to \infty$.
\end{thm}

\section{Dynamical Mahler Measure: Definition and Examples}
\label{sec:DMMdef}
Inspired by Definition~\ref{defn:MM} and the parallels between Mahler measure and the dynamical setting in equations~\eqref{eq:dynMah1D} and~\eqref{eq:mahlerheight}, we define the following:

\begin{defn} \label{def:dmm-defi}
Let $f \in \Z[z]$ be a monic polynomial of degree $d \geq 2$, and let $P \in \C(x_1, \ldots, x_n)^\times$. The \textbf{$f$-dynamical Mahler measure} of $P$ is the number
	\begin{equation}\label{eq:dmm-defi}
	 \m_f(P) := \int_{\J_f} \cdots \int_{\J_f} \log|P(z_1, \ldots, z_n)|\ d\mu_f(z_1) \cdots d\mu_f(z_n). 
	 \end{equation}
Note that as $\mu_f$ is a probability measure, the value of this integral is not affected by omitted variables, so in this sense the value of $\m_f$ is independent of $n$.
\end{defn}

It is not clear \emph{a priori} that the integral in~\eqref{eq:dmm-defi} converges---it does, and we prove this in Proposition~\ref{prop:convergence}---but before discussing these details we provide some examples where the dynamical Mahler measure can be explicitly computed.  The following lemmas will prove useful throughout. 

\begin{lem}
\label{products}
If $f \in \Z[z]$ and $P, Q \in \C(x_1, \ldots, x_n)^\times$, then \[ \m_f(PQ) = \m_f(P) + \m_f(Q). \]
\end{lem}
\begin{proof}
This follows immediately from the corresponding fact about logarithms.
\end{proof}

\begin{lem}
If $f$ and $g$ are nonlinear polynomials that commute under composition, then $\m_f = \m_g$.
\end{lem}
\begin{proof}
If $f$ and $g$ commute, then they have the same Julia set (see~\cite{AtelaHu}), and the equilibrium measure is determined by this set.
\end{proof}

\begin{prop}
\label{power-function-mahler-measure}
If $f(z) = z^d$ with $d \geq 2$, then $\m_f(P) = \m(P)$ for $P \in \C(x)^\times$.
\end{prop}
\begin{proof}
In this case, we have seen that $\J_f$ is given by the circle  $\TT^1$, and Theorem~\ref{equilibrium-measure-sequence} tells us that the equilibrium measure is the uniform measure on the circle:
\[
\frac{\chi_{\TT^1}dz}{2\pi i z},
\]
where $\chi_{\TT^1}$ is the characteristic function on the unit circle. Taking $z = e^{i \theta}$, we then have
	\[ \m_f(P) = \frac{1}{2\pi} \int_0^{2\pi} \log |P(e^{i\theta})|\ d\theta = \frac{1}{2\pi i} \int_{\TT^1} \log |P(z)| \frac{dz}{z} = \m(P). \qedhere \]
\end{proof}

\begin{prop}
\label{chebyshev-polynomial-mahler-measure}
Define the $d^\text{th}$ Chebyshev polynomial to be the polynomial $T_d(z) \in \Z[z]$ that satisfies
\begin{equation}
T_d(z+z^{-1}) = z^d + z^{-d}.
\label{eqn:chebyshev def}
\end{equation}
Then \[\m_{T_d}(P) = \m(P \circ w)\] for $P \in \C(x)^\times$, where $w(z) = z + z^{-1}$.
\end{prop}
\begin{proof}
Note that  $T_d \circ w = w \circ f$, where $f(z) = z^d$ (and note the analogy with Proposition~\ref{prop:conjugation} below). The function $w$ maps the Julia set $\J_f$ onto the Julia set $\J_{T_d}$, so that $\J_{T_d}$ is the segment $[-2, 2]$ (traversed twice as $z$ proceeds around the unit circle). It follows from Theorem~\ref{equilibrium-measure-sequence} that the equilibrium measure on $\J_{T_d}$ is the pushforward $w_\ast \mu_f$ of the equilibrium measure on $\J_f$, so
	\begin{align*}
	\m_f(P) &= \int_{\J_{T_d}} \log |P(z)|\ d\mu_{T_d}(z) \\
	&= \int_{w(\J_f)} \log |P(z)|\ dw_\ast \mu_f(z) \\
	&= \int_{\J_f} \log |P(w(z))|\ d\mu_f(z) \\
	&= \m_f(P \circ w) \\
	&= \m(P \circ w)
	\end{align*}
by Proposition~\ref{power-function-mahler-measure}.
\end{proof}

\begin{remk}
Incidentally, writing $u = z + z^{-1} = e^{i\theta} + e^{-i\theta} = 2 \cos \theta$, we have $du = -2 \sin \theta\ d\theta$, which is to say
	\[ - \frac{du}{\sqrt{4 - u^2}} = d\theta. \]
Thus we can write
	\begin{align*}
	\m_{T_d}(P) &= \frac{1}{\pi} \int_0^{\pi} \log |P(e^{i\theta} + e^{-i\theta})|\ d\theta \\
	&= \frac{1}{\pi} \int_{-2}^2 \log |P(u)| \frac{du}{\sqrt{4 - u^2}},
	\end{align*}
from which we see that the equilibrium measure on the segment $[-2, 2]$ is
	\[ \frac{\chi_{[-2,2]}dz}{\pi \sqrt{4-z^2}}. \]
\end{remk}

More generally, we have:
\begin{prop}\label{prop:gen-cheby}
Let $\alpha, \beta \in \C$, and let  $f(z)=\frac{\beta-\alpha}{4}T_d\left(\frac{4z-2(\alpha+\beta)}{\beta-\alpha}\right)+\frac{\alpha+\beta}{2}$ where $T_d$ is a Chebyshev polynomial defined in equation~\eqref{eqn:chebyshev def} for $d\geq 2$. For $P \in \C[x]$, we have 
\begin{equation}\label{eq:mfchange}
\m_f(P)=\m\left(P\circ \left(\frac{\beta-\alpha}{4}(z+z^{-1})+\frac{\alpha+\beta}{2}
 \right) \right).
 \end{equation}
\end{prop}
\begin{proof} A  change of variables shows that the Julia sets of these polynomials are given by 
 $\J_f=[\alpha,\beta]$, where this is to be understood as the  line segment connecting $\alpha$ and $\beta$ in the complex plane. The equilibrium measure is then given by  
\[\frac{\chi_{[\alpha,\beta]}dz}{\pi \sqrt{(z-\alpha)(\beta-z)}},\]
where $\chi_{[\alpha,\beta]}$ is the characteristic function on the segment $[\alpha, \beta]$.
This gives  
\begin{align*}
 \m_f(P) = & \frac{1}{\pi} \int_{\alpha}^\beta \log|P(z)|\frac{dz}{\sqrt{(z-\alpha)(\beta-z)}}.
\end{align*}
Setting $z=\frac{\beta-\alpha}{2}\cos(\pi \theta)+\frac{\alpha+\beta}{2}$ gives 
\begin{align*}
\m_f(P)=&\int_{0}^1 \log\left|P\left(\frac{\beta-\alpha}{2}\cos(\pi \theta)+\frac{\alpha+\beta}{2}\right)\right| d\theta\\=&\frac{1}{2}\int_{-1}^1 \log\left|P\left(\frac{\beta-\alpha}{2}\cos(\pi \theta)+\frac{\alpha+\beta}{2}\right)\right|  d\theta.
 \end{align*}
Substituting $w=e^{i\theta}$, we turn the domain of integration into the unit circle, and we conclude that $\m_f$ is given by \eqref{eq:mfchange}. 
\end{proof}

We provide the details of Proposition~\ref{prop:gen-cheby} because it is so difficult, in general, to calculate dynamical Mahler measure exactly. However, we note that the result can also be viewed as a consequence of Proposition~\ref{chebyshev-polynomial-mahler-measure} about Chebyshev polynomials and the following result on dynamical Mahler measure for conjugate maps.

\begin{prop}\label{prop:conjugation}
Let $f \in \C[z]$ and $P \in \C[z]$, and let $L(z) = az+b \in \C[z]$ with $a\neq 0$. Then
\[
\m_{f^L}(P) = \m_f(P \circ L^{-1}).
\]
\end{prop}

\begin{proof}
Note first that the Julia sets of $f$ and $f^L$ are related by
	\[ \J_{f^L} = L^{-1}(\J_f). \]
Next, fixing $w \in \J_f$, observe that by Theorem~\ref{equilibrium-measure-sequence}, the measure $\mu_{f^L}$ is the limit in the weak$^\ast$ topology of the sequence of measures
	\[ d^{-n} \sum_{(f^L)^n(\zeta) = w} \delta_\zeta = d^{-n} \sum_{f^n(L(\zeta)) = L(w)} (L^{-1})_\ast \delta_{L(\zeta)},\]
where $d$ denotes the degree of $f$ and $(L^{-1})_\ast \delta_{L(\zeta)}$ denotes the measure defined by $(L^{-1})_\ast \delta_{L(\zeta)}(X) = \delta_{L(\zeta)}(L(X))$. 
But by Theorem~\ref{equilibrium-measure-sequence}, this sequence also has weak$^\ast$-limit $(L^{-1})_\ast \mu_f$. So $\mu_{f^L} = (L^{-1})_\ast \mu_f$. Making the substitution $w = L(z)$, we then have
	\begin{align*}
	\m_{f^L}(P) &= \int_{\J_{f^L}} \log |P(z)|\ d\mu_{f^L}(z) \\
	&= \int_{L^{-1}(\J_f)} \log |P(z)|\ d(L^{-1})_\ast \mu_f(z) \\
	&= \int_{\J_f} \log |P(L^{-1}(w))|\ d\mu_f(w) \\
	&= \m_f(P \circ L^{-1}),
	\end{align*}
as desired.
\end{proof}

\section{Dynamical versions of classical results}\label{sec:PrevPaper}
In this section, we summarize results from~\cite{TwoVarPoly}, focusing on the connections between these results and classical Mahler measure as outlined in Section~\ref{sec:ClassicalMahler}. We provide more detail and refine some of these results in the next sections.

\noindent
{\bf Jensen's formula} gave us an equivalent definition of Mahler measure that extended naturally to higher dimensions:

\[
\m\left( P\right) = \log|a| + \sum_{|\alpha_i|>1} \log |\alpha_i|
= \frac{1}{2\pi i}\int_{\TT^1}  \log \left| P\left(z \right)\right| \frac{dz}{z}.
\]

\noindent
{\bf Dynamical Jensen's formula}~\cite[Lemma 3.1] {TwoVarPoly} plays a similar role, allowing us to prove that the integral in \eqref{eq:dmm-defi} converges, and that when $P\in \Z[x_1,\dots,x_n]$, we have   $\m_f(P)\geq 0$. Here $p_{\mu_f}$ is the potential function (see Definition~\ref{def:DefPotential} and Proposition~\ref{dynamical-jensen}).

\[
\m_f(P) = \log|a| + \sum_{i} p_{\mu_f}(\alpha_i) = \log|a| +  \int_{\K_f} \log|z-w|\ d\mu_f(w).
\]

\bigskip
\noindent
{\bf  Kronecker's Lemma} tells us which integer polynomials have Mahler measure zero:
Let $P \in \Z[x]$. If  $\m(P)=0$, then the roots of $P$ are either zero or roots of unity. Conversely, if $P$ is primitive and its roots  either zero or roots of unity, then $\m(P)=0$. 

\bigskip
\noindent
{\bf Dynamical Kronecker's Lemma}~\cite[Lemmas 1.2 and 4.3]{TwoVarPoly}  answers the same question for dynamical Mahler measure. Recalling the driving analogy of arithmetic dynamics, that preperiodic points are like torsion points in arithmetic geometry, the result feels natural.
\begin{lem}[Dynamical Kronecker's Lemma] \label{lem:dynamicalKronecker} Let $f \in \Z[z]$ be monic 
of degree $d \ge 2$  and let $P \in \Z[x]$. 
Then we have $\m_f(P)=0$ if and only if $P(x) = \pm \prod_i (x-\alpha_i)$ with each $\alpha_i$ a preperiodic point of $f$.
\end{lem}

\bigskip
\noindent
{\bf The Boyd-Lawton Theorem} relates single-variable and multivariate Mahler measure.
For $P \in \C(x_1, \dots, x_n)^\times$,
\[\lim_{k_2 \rightarrow \infty}\dots \lim_{k_n \rightarrow \infty}\m(P(x, x^{k_2}, \dots, x^{k_n})) = \m(P(x_1,\dots, x_n))\]
with  $k_2, \dots, k_n \rightarrow \infty$ independently from each other.

\bigskip
\noindent
{\bf The Weak Dynamical Boyd-Lawton Theorem}~\cite[Proposition 1.3]{TwoVarPoly} provides a partial analogue in the dynamical setting for polynomials in two variables.

\begin{prop}[Weak Dynamical Boyd-Lawton]\label{prop:weakdynamicalBL}
	Let $f \in \Z[z]$ monic 
of degree $d \ge 2$  and let $P \in \C[x , y]$. Then
	\[
	\limsup_{n\to\infty} \m_f(P(x, f^n(x))) \le \m_f(P(x, y)).
	\]
\end{prop}

\bigskip
\noindent
{\bf Lehmer's Question}
asks if there are integer polynomials with arbitrarily small Mahler measure, or if the Mahler measure of $P \in \Z[x]$ with $\m_f(P) \neq 0$ is bounded away from zero.

 \bigskip
 \noindent
 {\bf Dynamical Lehmer's Conjecture}~\cite[Conjecture~3.25]{Silverman-arithmetic-dynamical} asks the same question for dynamical Mahler measure.

\begin{conj}[Dynamical Lehmer's Conjecture] \label{conj:dynamicalLehmer}
 There is some $\delta = \delta_f > 0$ such that any  single-variable polynomial $P\in \Z[x]$ with $\m_f(P) > 0$ satisfies $\m_f(P) > \delta$.
\end{conj}

\bigskip
\noindent
{\bf Higher dimensional Kronecker's Lemma}  could be stated quite simply, and had a similar feel to the one-dimensional version:

 \bigskip
\begin{thmn}[\ref{thm:HigherdKronecker}]\cite[Theorem 3.10]{Everestward} 
For any primitive polynomial $P\in \Z[x_1,\dots,x_n]$, $\m(P) =0$ if and only if $P$ is the product of a monomial and cyclotomic polynomials evaluated on monomials. 

\end{thmn}

\bigskip
The main result of~\cite{TwoVarPoly} provides a partial two-variable Kronecker's Lemma for dynamical Mahler measure, but the statement and hypotheses are significantly more delicate than in the one-variable case.

 \begin{thm}\cite[Theorem 1.5]{TwoVarPoly} \label{thm:mainresult}
 	Assume the Dynamical Lehmer's Conjecture.
 	
 	Let $f \in \Z[z]$ be a monic polynomial of degree $d \ge 2$ which is not conjugate to $z^d$ or to $\pm T_d(z)$, where $T_d(z)$ is the $d^\text{th}$ Chebyshev polynomial.
 	Then any polynomial $P \in \Z[x, y]$ which is irreducible in $\Z[x,y]$  (but not necessarily irreducible in $\C[x, y]$)
 	with $\m_f(P) = 0$ and which contains both variables $x$ and $y$
 	divides a product of complex polynomials of  the following form:
 	 \[	\tilde{f}^n(x)  - L(\tilde{f}^m(y)),\]
  where $m, n \ge 0$ are integers, $L \in \C[z]$ is a linear polynomial commuting with an iterate of $f$, and $\tilde{f} \in \C[z]$ is a non-linear polynomial of minimal degree commuting with an iterate of $f$ (with possibly different choices of $L$, $\tilde{f}$, $n$, and $m$ for each factor).

  As a partial converse, suppose there exists a product of complex polynomials $F_j$ such that
		\begin{enumerate}
		\item each $F_j$ has the form $\tilde{f}^{n}(x) - L(\tilde{f}^{m}(y))$, where $L$ and $\tilde{f}$ are as above (with possibly different choices of $L$, $\tilde{f}$, $n$, and $m$ for each $j$);
		\item $\prod F_j \in \Z[x, y]$; and
		\item $P$ divides $\prod F_j$ in $\Z[x, y]$.
		\end{enumerate}
	Then $\m_f(P) = 0$.
  
  \end{thm}
  
  In Section~\ref{sec:DynKronLem}, we discuss several other statements that are either known or conjectured to be equivalent to having dynamical Mahler measure~0, and in Section~\ref{sec:e1e2} we prove a new implication in this family of results. In Section~\ref{sec:ad}, we prove a new version of Theorem~\ref{thm:mainresult} in which we replace the assumption of Dynamical Lehmer's Conjecture with the assumption that $\PrePer(f) \subseteq \J_f$. This is a strengthening of the result in some respects, since the hypothesis on the preperiodic points is much easier to check, when it holds, than Dynamical Lehmer's Conjecture. 
However, there are certainly polynomials $f\in \Z[z]$ for which that assumption does not hold. This is discussed in Section~\ref{sec:J=K}.

\section{Convergence and Positivity}\label{sec:DynMahlerConverge}

In this section we give an introduction to potentials, and then use them to prove the existence of the dynamical Mahler measure.

\begin{defn}\label{def:DefPotential}
The \textbf{potential} of a finite Borel measure $\mu$ with compact support $K$ is the function $p_\mu: \C \to [-\infty, \infty)$ given by
	\[ p_\mu(z) = \int_K \log|z - w|\ d\mu(w). \]
\end{defn}

We can see the relationship between potentials and dynamical Mahler measure in the following result, which should be considered the dynamical analogue of Jensen's formula:

\begin{prop}
\label{dynamical-jensen}
Suppose $P(x)$ factors over $\C$ as $P(x) = a \prod_i (x - \alpha_i)$. Then
	\[ \m_f(P) = \log|a| + \sum_i p_{\mu_f}(\alpha_i). \]
\end{prop}
\begin{proof}
We have
	\begin{align*}
	\m_f(P) &= \int_{\J_f} \log \left| a \prod_i (z - \alpha_i) \right|\ d\mu_f \\
		&= \log|a| + \sum_i \int_{\J_f} \log|z - \alpha_i|\ d\mu_f \\
		&= \log|a| + \sum_i p_{\mu_f}(\alpha_i). \qedhere
	\end{align*}
\end{proof}

\begin{remk}
If $f$ is a monic polynomial, the potential $p_{\mu_f}$ of the equilibrium measure on its Julia set is equal to the \emph{Green's function} $g_{F_\infty}(z, \infty)$ on $F_\infty$, the complement of the filled Julia set $\K_f$ in the Riemann sphere. (See~\cite[Theorem~6.5.1]{Ransford} and the proof of~\cite[Theorem~4.4.2]{Ransford}.)
\end{remk}

Let us make a few more observations about potentials before returning to dynamical Mahler measure.

\begin{prop}
\label{potential-continuous}
Let $f \in \C[z]$ be a nonlinear polynomial. Then the potential $p_{\mu_f}$ is continuous.
\end{prop}
\begin{proof}
First, the potential is harmonic, and thus continuous, on $\C \setminus \K_f$~\cite[Theorem~3.1.2]{Ransford}.  As $F_\infty$ is a regular domain~\cite[Corollary~6.5.5]{Ransford}, we have $p_{\mu_f}(z) = I(\mu_f)$ for all $z \in \J_f$~\cite[Theorem~4.2.4]{Ransford}, and it is shown in the proof of~\cite[Corollary~6.5.5]{Ransford} that
	\[ \lim_{\substack{z \to \zeta \\ z \notin \K_f}} p_{\mu_f}(z) = I(\mu_f) \]
for all $\zeta \in \J_f$. Finally, it follows from Frostman's Theorem~\cite[Theorem~3.3.4]{Ransford} that $p_{\mu_f}(z) = I(\mu_f)$ on the interior of $\K_f$ also.
\end{proof}

\begin{prop}
\label{potential-nonnegative}
Let $f \in \C[z]$ be a nonlinear, monic polynomial. Then $p_{\mu_f}(z) \geq 0$ for all $z \in \C$, and $p_{\mu_f}(z) = 0$ if and only if $z \in \K_f$.
\end{prop}
\begin{proof}
It follows from~\cite[Theorem~6.5.1]{Ransford} that $I(\mu_f) = 0$ if $f$ is monic. The  proof of Proposition~\ref{potential-continuous} then shows that $p_{\mu_f}(z) = 0$ for $z \in \K_f$, while $p_{\mu_f}(z) > 0$ for $z \notin \K_f$~\cite[Theorem~4.4.3]{Ransford}.
\end{proof}

We now return to the dynamical Mahler measure.

\begin{prop}\label{prop:convergence}
Let $f \in \Z[z]$ be a monic, nonlinear polynomial, and let $P \in \C(x_1, \ldots, x_n)^\times$. Then the integral defining the $f$-dynamical Mahler measure of $P$ converges, and if $P$ is furthermore a nonzero integer polynomial, then $\m_f(P) \geq 0$.
\end{prop}
\begin{remk}
 This result appears in \cite[Proposition 3.2]{TwoVarPoly}. In the interest of providing a self-contained introduction to the key ideas in the subject, we present here an expanded and more detailed proof. The argument is based on the proof of~\cite[Lemma~3.7]{Everestward} for the classical Mahler measure.
\end{remk}

\begin{proof}
It suffices to consider the case of $P$ a polynomial, since $\m_f(F/G) = \m_f(F) - \m_f(G)$ by Lemma~\ref{products}.

We induct on the number of variables. When $n$ = 1, we can factor $P$ over $\C$ as $a \prod_i (x - \alpha_i)$. By Proposition \ref{dynamical-jensen}, we have
	\[ \m_f(P) = \log|a| + \sum_i p_{\mu_f}(\alpha_i). \]
Since the potential $p_{\mu_f}$ is nonnegative on $\C$, we can immediately conclude that the integral defining $\m_f(P)$ converges and that it is nonnegative when $P$ has integer coefficients.

Now assume the result holds for polynomials in $n - 1$ variables, and let $P \in \C[x_1, \ldots, x_n]$. Write $P$ as a polynomial in $x_1$ with coefficients in $\C[x_2, \ldots, x_n]$:
	\[ P(x_1, \ldots, x_n) = a_d(x_2, \ldots, x_n) x_1^d + \dots + a_0(x_2, \ldots, x_n). \]
Factor this as
	\[ a_d(x_2, \ldots, x_n) \prod_{j=1}^d (x_1 - g_j(x_2, \ldots, x_n)) \]
for some algebraic functions $g_j$. We then have
	\begin{align}
\nonumber	\m_f(P) &= \m_f(a_d) + \int_{\J_f} \cdots \int_{\J_f} \log \left| \prod_{j=1}^d (z_1 - g_j(z_2, \ldots, z_n)) \right|\ d\mu_f(z_1) \cdots d\mu_f(z_n) \\
	&= \m_f(a_d) + \int_{\J_f} \cdots \int_{\J_f} \sum_{j=1}^d p_{\mu_f}(g_j(z_2, \ldots, z_n)) \ d\mu_f(z_2) \cdots d\mu_f(z_n). \label{eq:non-neg}
	\end{align}
By the induction hypothesis, $\m_f(a_d)$ exists and is nonnegative if $P$, and thus $a_d$, has integer coefficients. While the $g_j$ may not be continuous, the multiset of values $\{ g_j(z_2, \ldots, z_n) \}$ is, so it follows from Propositions~\ref{potential-continuous} and \ref{potential-nonnegative} that the integrand
	\[ \sum_{j=1}^d p_{\mu_f}(g_j(z_2, \ldots, z_n)) \]
is nonnegative and continuous away from any poles of the $g_j$. On the other hand, as $\J_f$ is compact, the polynomial $P(z_1, \ldots, z_n)$ is bounded above on $\J_f^n$, so the integral defining $\m_f(P)$ is also. The same can then be said for the integral
	\[ \int_{\J_f} \cdots \int_{\J_f} \sum_{j=1}^d p_{\mu_f}(g_j(z_2, \ldots, z_n)) \ d\mu_f(z_2) \cdots d\mu_f(z_n) \]
by the finiteness of $\m_f(a_d)$; it follows that this integral converges, despite the presence of any poles of the $g_j$, and therefore the integral defining $\m_f(P)$ does also.
\end{proof}

\section{Multivariable Analogues of Dynamical Kronecker's Lemma}
\label{sec:DynKronLem}

Multivariate dynamical Mahler measure was  defined in~\cite{TwoVarPoly}, but similar ideas have appeared in the literature in recent years. In this section, we present a summary of some  results in arithmetic dynamics.  These statements, which include the statement that a polynomial has dynamical Mahler measure zero, are all known or conjectured to be equivalent.

Assume as usual that $f \in \Z[x]$ is monic of degree $d$, and $P \in \Z[x_1, \dotsc, x_n]$.  As in the  statement of Theorem~\ref{thm:mainresult}, $L$ always denotes a linear polynomial in $\C[z]$ commuting with an iterate of $f$, and $\tilde{f}$ always denotes a non-linear polynomial in $\C[z]$ of minimal degree commuting with an iterate of $f$.

\begin{itemize}
	\item[(a)] $\m_f(P) =0$.
	\item[(b)] $h(\overline{\{P=0\}}\subseteq X) = 0$, where  $\overline{\{P=0\}}$ is the Zariski closure of the hypersurface $\{P = 0\} \subseteq\A^n \subseteq X$, $X$ is either $(\PP^1)^n$ or $(\PP^n)$, and  $h$ is a dynamical height for subvarieties of $X$ of the type introduced in \cite{Zhang}.
\item [(c)] The hypersurface $\{P = 0\}\subseteq\A^n(\C)$ is preperiodic under the map $(x_1, \dotsc, x_n) \mapsto (f(x_1), \dotsc, f(x_n))$.
(A subvariety $V$ of a variety $X$ is preperiodic for a map $\Phi: X \to X$ if $\Phi^m (V) =  \Phi^n(V)$ for some $m \ne n$.)

\item[(d)] The hypersurface $\{P = 0\}\subseteq\A^n(\C)$ contains a Zariski dense subset of points that are preperiodic for the map $(x_1, \dotsc, x_n) \mapsto (f(x_1), \dotsc, f(x_n))$ (equivalently, with all coordinates preperiodic for $f$).
\item[(e1)] $P$ is primitive (gcd of coefficients $= 1$) and, inside the ring $\C[x_1, \dotsc, x_n]$, $P(x)$ divides some polynomial in $\C[x_1, \dotsc, x_n]$ which is a product of factors of the form $\tilde{f}^n(x_i)  - L(\tilde{f}^m(x_j))$ (here $i$ can equal $j$).
\item[(e2)] Inside the ring $\Z[x_1, \dotsc, x_n]$, $P(x)$ divides some polynomial in $\Z[x_1, \dotsc, x_n]$ which is a product of factors of the form $\tilde{f}^n(x_i)  - L(\tilde{f}^m(x_j))$ (the factors do not need to be in $\Z[x]$, but the product does). 
\end{itemize}
The known relationships among these statements are summarized in Figure~\ref{fig:diagram}.

\begin{figure}[h]
\begin{tikzpicture}
	[node distance = 2cm,
	known/.style={thick, black},
	conditional/.style={thick, blue},
	conjectural/.style={thick, red, dashed}]
\node (a) {(a)};
\node [right=of a] (e2) {(e2)};
\node [below=of a] (d) {(d)};
\node [below=of e2] (e1) {(e1)};
\node [below=of e1] (c) {(c)};
\node [below=of d] (b) {(b)};
\draw [->, known] (e2) -- node[above]{\cite{TwoVarPoly}} (a);
\draw [->, conditional] (a) -- node[left, black, align=right]{\cite{TwoVarPoly}, \\ Theorem~\ref{mahler-preperiodic}} (d);
\draw [->, known] (d) -- node[left]{\cite{Zhang}} (b);
\draw [<->, known] (e2) -- node[right]{Theorem~\ref{e1-implies-e2}} (e1);
\draw [<->, known] (d) -- node[above]{\cite{Ghioca-Nguyen-Ye-two-var}} (e1);
\draw [<->, known] (d) -- node[fill = white]{\cite{Ghioca-Nguyen-Ye-two-var}} (c);
\draw [<->, known] (e1) -- node[right]{\cite{MedvedevScanlon}} (c);
\draw [->, known] ([yshift=1mm] c.west) -- ([yshift=1mm] b.east); 
\draw [->, conjectural] ([yshift=-1mm] b.east) -- node[below, black]{\cite{Zhang}} ([yshift=-1mm] c.west);

\node (k) [right=of e1]
	{
	\begin{tabular}{rl}
	{\tikz [baseline] \draw [->, known] (0,0.5ex) -- (0.5, 0.5ex);} & Known implication \\
	{\tikz [baseline] \draw [->, conditional] (0,0.5ex) -- (0.5, 0.5ex);} & Conditional on some assumptions \\
	{\tikz [baseline] \draw [->, conjectural] (0,0.5ex) -- (0.5, 0.5ex);} & Conjectural \\
	\end{tabular}
	};
\end{tikzpicture}
\caption{The known relationships between statements (a--e), with references.}
\label{fig:diagram}
\end{figure}
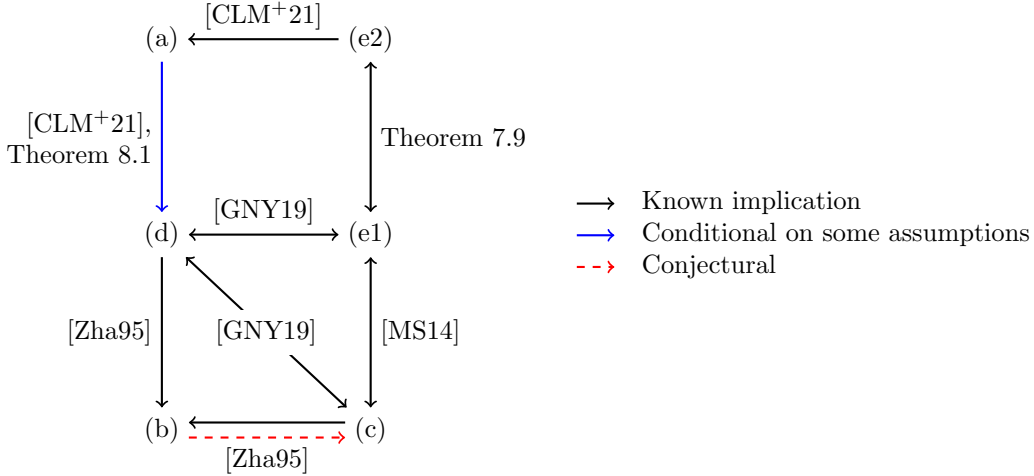
\subsection{Subvarieties with many preperiodic points and preperiodic subvarieties}

Historically, one of the first of these properties to be studied was property (d), as a special case of the following more general question in the field of unlikely intersections:
For an algebraic variety $X$ with a self map $\Phi: X \to X$, which subvarieties $Y$ of $X$ contain a Zariski dense subset of preperiodic points for $\Phi$?

This question was raised by Zhang~\cite{Zhang}, who conjectured that $Y$ has such a subset if and only if $Y$ is preperiodic for $\Phi$.  This conjecture is a generalization of the Manin--Mumford conjecture (proved by Raynaud ~\cite{Raynaud1,Raynaud2}) on subvarieties of abelian varieties containing infinitely many torsion points, and so Zhang in ~\cite{Zhang-distributions} calls this the Dynamical Manin--Mumford Conjecture.

\begin{conj}[Dynamical Manin--Mumford Conjecture] For any variety $X$ and dominant map $\Phi: X \to X$, a subvariety $Y$ of $X$ contains a Zariski dense subset of preperiodic points if and only if $Y$ is preperiodic. 
\end{conj}

In terms of our diagram, this is saying that (d) $\iff$ (c).  However most of the study of this conjecture has been focused on the implication (d) $\implies$ (c), which has generally been the harder direction.

The Dynamical Manin--Mumford Conjecture has been studied in various contexts, and is now known not to hold in full generality as originally stated (counterexamples, and a refined statement, have been given in \cite{Ghioca-Tucker-Zhang}).
However, in the case of interest for our application, it is known to be true:

\begin{thm}[Ghioca, Nguyen, Ye~\cite{Ghioca-Nguyen-Ye-two-var, Ghioca-Nguyen-Ye-multivar}]\label{GNY-general}
	If $\Phi: (\PP^1)^n \to (\PP^1)^n$ is of the form $f \times \dotsb \times f$, where $f$ is a non-exceptional rational map (not conjugate to a power map, a Chebyshev polynomial, or a Latt\`es map), then the Dynamical Manin--Mumford conjecture holds for the pair $((\PP^1)^n, \Phi).$
\end{thm}	
Note that although this theorem is for $(\PP^1)^n$, the result also holds for the restriction to $\A^n$, since $\A^n$ is Zariski dense in $(\PP^1)^n$.
Ghioca, Nguyen, and Ye actually show a more general statement, which includes the case $\Phi = f_1 \times \dotsb \times f_n$ where $f_1, \dotsc, f_n$ are non-exceptional and all of the same degree.
 Dujardin and Favre \cite{Dujardin-Favre} have shown a related result:  that Dynamical Manin--Mumford holds for $(\A^2, \Phi)$, where $\Phi$ is any automorphism of H\'enon type. 

The set of invariant subvarieties for maps $\Phi: \A^n \to \A^n$ of the form $(x_1, \dotsc, x_n) \mapsto (f_1(x_1), \dotsc, f_n(x))$ was first determined by Medvedev and Scanlon~\cite{MedvedevScanlon}, and their work can be  extended to give all preperiodic subvarieties. 
We are only interested in the case where $f_1 = \dotsb = f_n$ and of preperiodic hypersurfaces.  (However, it turns out that $f_1 = \dotsb = f_n$ is the most interesting case, and also that all lower-dimensional preperiodic subvarieties are generated as intersections of preperiodic hypersurfaces.)  We compile the relevant results in the theorem statement below:

\begin{thm}[Medvedev, Scanlon~\cite{MedvedevScanlon}]
\label{thm:MedScan}
		If $\Phi: (\PP^1)^n \to (\PP^1)^n$ is of the form $f \times \dotsb \times f$, where $f$ is a non-exceptional rational map (not conjugate to a power map, a Chebyshev polynomial, or a Latt\`es map), then any hypersurface in $\A^n$ that is preperiodic for $\Phi$ is of the form $\{P = 0\}$ where $P(x)$ divides some polynomial in $\C[x_1, \dotsc, x_n]$ which is a product of factors of the form $\tilde{f}^n(x_i)  - L(\tilde{f}^m(x_j))$.
\end{thm}

In terms of Figure~\ref{fig:diagram}, Theorem~\ref{thm:MedScan} says that (c) implies (e1).

The proof of Dynamical Manin--Mumford by Ghioca, Nguyen, and Ye similarly proceeds by explicitly describing the subvarieties of $(\PP^1)^n$ that have infinitely many preperiodic points for $\Phi$, proving that (d) implies (e1) in Figure~\ref{fig:diagram}.  They combine this with an argument for (c) implies (d) to give an independent proof of (c) implies (e1).

\subsection{The Connection with Heights}
From equations~\eqref{eq:dynMah1D} and~\eqref{eq:mahlerheight}, we see that single-variable  Mahler measure is directly related to heights of points in $\PP^n(\overline{\Q})$. It is natural to ask if multivariate Mahler measure is also given by a height.

A very strong candidate for this is the dynamical height of subvarieties introduced by Zhang in~\cite{Zhang}. Zhang shows that this dynamical height vanishes on both preperiodic subvarieties and on subvarieties with infinitely many preperiodic points, and he conjectures the converse.  Since dynamical Mahler measure also detects polynomials whose zero locus is preperiodic or has infinitely many preperiodic points, it is natural to conjecture that $\m_f(P)$ is equal to the dynamical height of the hypersurface $\{P = 0\}$ with respect to the map $f \times \dotsb \times f$.
This conjecture is also supported by work of Chambert-Loir, and Thuillier~\cite{Chambert-Loir-Thuillier} showing that  ordinary multivariate Mahler measure agrees with the dynamical height for hypersurfaces in $\PP^n$ under the map $f \times \dotsb \times f$, where $f$ is a power map.
 
 A promising direction for future work would be to prove this relationship between dynamical Mahler measure and dynamical height, which would provide an additional connection between (a) and (b)  in Figure~\ref{fig:diagram}.  Zhang's height is too technical to define here, but we give a general overview of the main ideas.

Let $X$ be a projective variety with a line bundle $\mathcal{L}$, and $\Phi: X \to X$ a dominant endomorphism (that acts compatibly with the line bundle: $\Phi^* \mathcal{L} \cong \mathcal{L}^d$, where $d$ should be thought of as the degree of $\Phi$).   Zhang defined a height $h_\Phi$  on subvarieties of $X$ (or more generally, cycles on $X$) which, like the dynamical height we saw earlier, behaves nicely under pushforward: $h_{\Phi}(\Phi(Y)) = d h_{\Phi}(Y)$.

The dynamical height of a subvariety is always non-negative.  If $Y$ is preperiodic, then a formal consequence of the compatibility with pushforward is that $h_\Phi(Y) = 0$ (\cite[Theorem 2.4(b)]{Zhang}). The converse implication is conjectured~\cite[Conjecture 2.5]{Zhang}.  
Zhang also  shows that if $h_\Phi(Y) > 0$, then there is a Zariski open subset $U \subseteq Y$ which contains no preperiodic subvarieties, hence in particular no preperiodic points. So the preperiodic points of $Y$ are not Zariski dense.  Taking the contrapositive, we have the following result:
\begin{prop}[Zhang~\cite{Zhang}]\label{general d implies b}
	If the preperiodic points of $Y$ are Zariski dense, then $h_\Phi(Y) = 0$.
\end{prop}

In our situation, we are interested in polynomials $P(x_1, \dotsc, x_n) \in \C[x_1,\dots,x_n]$. These polynomials naturally cut out hypersurfaces in $\A^n$, not projective varieties.  However, we can solve this problem by completing $\A^n$ to a projective variety: either to $\PP^n$ (as in the work of Chambert-Loir and Thullier~\cite{Chambert-Loir-Thuillier}), or to $(\PP^1)^n$ (as is the work of Ghioca, Nguyen, and Ye~\cite{Ghioca-Nguyen-Ye-two-var, Ghioca-Nguyen-Ye-multivar}).  We conjecture that the heights given by the two options are equal to each other, as well as to the dynamical Mahler measure.

In terms of Figure~\ref{fig:diagram},  \cite[Theorem 2.4(b)]{Zhang} gives the implication (c) implies (b),  \cite[Conjecture 2.5]{Zhang} would give (b) implies (c), and Proposition~\ref{general d implies b} gives (d) implies (b).

\subsection{New results}
Here we highlight the equivalences and implications from Figure~\ref{fig:diagram} proved  in~\cite{TwoVarPoly} and in Sections~\ref{sec:e1e2} and~\ref{sec:ad} of the current work.

(a) implies (d): This result for two-variable polynomials, conditional on the Dynamical Lehmer Conjecture, is the  main result of~\cite{TwoVarPoly}. (See  Theorem \ref{thm:mainresult} here for a complete statement.)
The proof proceeds in two main steps, first obtaining (a) implies (d) conditional on the Dynamical Lehmer Conjecture, and then using the  two-variable case of Theorem~\ref{GNY-general}  from~\cite{Ghioca-Nguyen-Ye-two-var} for the (d) implies (e) step.
More specifically, in \cite[Propostion 7.6]{TwoVarPoly} we show that (a) implies (d) conditional on Dynamical Lehmer for pairs $(P, f)$ satisfying a technical condition that we call the bounded orders property, which we then show holds in all cases.

In Section 8 of this paper, we strengthen the two-variable result by showing that (a) implies (d) without Dynamical Lehmer's conjecture for polynomials $f$ such that $\PrePer(f) \subseteq\J(f)$.  Again, combining this with the implication (d) implies (e) from \cite{Ghioca-Nguyen-Ye-two-var} gives us a two-variable Kronecker's Lemma for these polynomials. The implication (a) implies (d) for polynomials in more than two variables is still open.

(e2) implies (a): This was shown in ~\cite[Corollary 6.4]{TwoVarPoly}. The strategy for the proof consists of noticing that $\m_f(x-y)=0$ for arbitrary $f\in \Z[x]$ monic, and then using the fact that the Mahler measure is invariant under composition with any polynomial commuting with $f$. 

(e1) implies (e2): This is shown in Section \ref{sec:e1e2} by studying integrality of polynomials that satisfy certain commutative properties. A key step is to prove that a polynomial commuting with some monic $f \in \overline{\Z}[x]$ (and satisfying certain technical conditions)  must have coefficients in $\overline{\Z}[x]$. (See Proposition \ref{prop:coeff}.)

\section{An integrality property of commuting polynomials} \label{sec:e1e2}

To show that (e1) implies (e2) in Figure~\ref{fig:diagram}, we need to know that we can choose our product of factors of the form $\tilde{f}^n(x_i)  - L(\tilde{f}^m(x_j))$ to have integer coefficients.  We will do this by showing that $\tilde{f}$ and $L$ have algebraic integer coefficients, hence the product also has algebraic integer coefficients, and the coefficients of the product can be then assumed to be rational integers by enlarging the set of factors, if necessary, to be stable under the Galois action.

Let $\overline{\Z}$ be the ring of algebraic integers, which has fraction field $\overline{\Q}$ and unit group $\overline{\Z}^\times$.
First we  show that $\tilde{f}\in \overline\Z[x]$.

\begin{lem}\label{lem:integercomposition}
	If $g, h \in \overline{\Q}[x]$ are polynomials of positive degree with leading coefficients in $\overline{\Z}^\times$ such that $f = g \circ h \in \Z[x]$, then the polynomials  $g(x + h(0))$ and $h(x) - h(0)$ both lie in $\overline{\Z}[x]$.
\end{lem}
\begin{proof}
	We follow the method of proof of~\cite[Theorem 2.1]{Gusic}.
	
	By replacing $g(x)$ and $h(x)$ with $g(x+ h(0))$ and $h(x) - h(0)$ respectively, we may assume that $h(0) = 0$.  In this case, it suffices to show that $g(x)$, $h(x)  \in \overline{\Z}[x]$.
	
	Write $g(x) = a \prod_i (x-\alpha_i)$ and $h(x) = b \prod_j (x-\beta_j)$.  By assumption $a, b \in \overline{\Z}^\times$.	
	Note that the polynomial $f = g \circ h$ has leading coefficient equal to $a \cdot b^{\deg g} \in \overline{\Z}^\times$.   Hence, the roots of $f$ are all algebraic integers, and we can factor $f(x)$ in $\overline{\Z}[x]$ as $f(x) = a \cdot b^{\deg g} \prod_k (x - \gamma_k)$ with $\gamma_k \in \overline{\Z}$.
	We have another factorization:	
	\[
	f(x)  = g(h(x)) = a \prod_{i} (h(x) - \alpha_i).
	\]
	By unique factorization, we must have $h(x) - \alpha_i = b \prod_{k \in S_i} (x - \gamma_{k})$ for some subset $S_i$ of the $\gamma_k$,  
	and so in particular $h(x) - \alpha_i \in \overline{\Z}[x]$ has algebraic integer coefficients.  Since we have assumed that $h(0) = 0$, we conclude that $\alpha_i \in \overline{\Z}$ and $h(x) \in \overline{\Z}[x]$.  Then also $g(x) = a \prod_i (x-\alpha_i) \in \overline{\Z}[x]$, and this concludes the proof. 
\end{proof}

\begin{prop}\label{prop:iterates}
	Suppose that $f \in \overline{\Q}[x]$ has degree $>1$ and that the $n$-fold iterate $f^n = f \circ \dotsb \circ f$ is a monic polynomial in $\overline{\Z}[x]$.  Then in fact $f \in \overline{\Z}[x]$.
\end{prop}
\begin{proof}
	Since the leading coefficient of $f^n$ is $1$, we see that $f$ must have leading coefficient in $\overline{\Z}^\times$ (in fact, it must be a root of unity, but we do not need this). The same is true for any iterate of $f$.
	
	Write $c = f(0)$.  We now apply Lemma~\ref{lem:integercomposition} to the composition $f^n = (f^{n-1}) \circ f$, and conclude that  $f(x) - c \in \overline{\Z}[x]$.
	
	Hence, if we write $f(x) = a_d x^d +  \dotsb  + a_1 x +  c$, we have shown that all $a_j$ with $j > 0$ are algebraic integers.  It remains to show that also $c$ is an algebraic integer. 	For this, we note that 

\begin{equation}\label{eq:slick}
	f^n(c) -c = f^{n+1}(0) - c = f(f^n(0)) - c = a_d (f^n(0))^d + \dotsb + a_1(f^n(0)),
	\end{equation}
	where the constant terms cancel out.   The right hand side lies in $\overline{\Z}$ because all the $a_i$ do, as does $f^n(0)$ since $f^n \in \overline{\Z}[x]$.
	
	Hence $c$ is a root of a polynomial of the form $f^n(x) - x - A = 0$, where $A \in \overline{\Z}$ is the right hand side of \eqref{eq:slick}.  Since $f$ has degree $>1$, this is a monic polynomial with algebraic integer coefficients, hence $c \in \overline{\Z}$ also, as desired.
\end{proof}

Before proving our next statement, we need the following result:

\begin{thm}\cite{Julia,Ritt} \label{thm:JR}
If two polynomials $f$ and $g$ commute under composition, then up to conjugation with the same linear polynomial, either both are power functions, both are plus or minus Chebyshev polynomials, or an iterate of one is equal to an iterate of the other.
\end{thm}

\begin{cor}\label{cor:gint}
	If $f\in \overline{\Z}[x]$ is a monic polynomial of degree $>1$ that is not conjugate (over $\C$) to a power function or plus or minus a Chebyshev polynomial,
and $g \in \overline{\Q}[x]$ of degree $>1$ commutes with some iterate of $f$, then in fact $g \in \overline{\Z}[x]$.  
\end{cor}
\begin{proof}
	By assumption, $g$ commutes with $f^k$ for some $k$.  It follows from Theorem~\ref{thm:JR} that $g^a = (f^k)^b$ for some positive integers $a$ and $b$.  Hence $g^a \in \overline{\Z}[x]$ and is monic. By Proposition \ref{prop:iterates}, we conclude that $g \in \overline{\Z}[x]$.
\end{proof}

\begin{remk}
It would be worth investigating if Corollary \ref{cor:gint} is still true even when $f$ is conjugate to a power function or plus or minus a Chebyshev polynomial. 
It would be also interesting to have a proof of the statement that does not rely on Theorem \ref{thm:JR}.
\end{remk}

We now show that the commuting linear function  $L$ has coefficients in the algebraic integers.

\begin{prop}
	If $f$ in $\overline{\Z}[x]$ is monic of degree $>1$, and $L \in \overline{\Q}[x]$ is a linear polynomial that commutes with $f$, then $L \in \overline{\Z}[x]$.
\end{prop}
\begin{proof}
	Write $f(x) = x^n + c_{n-1}x^{n-1} + \dotsb + c_1 x  + c_0$, and $L(x) = ax + b$.
First	 look at the leading coefficient of $L \circ f = f \circ L$: 
	\[
	a = a^n,
	\]
so $a$ is a root of unity, hence in $\overline{\Z}$.

Now look at the constant coefficient of $L \circ f = f \circ L$: 
    \[
    a c_0 + b = f(b),
    \]
	so $b$ is a root of the equation $f(x) - x - a c_0 \in \overline{\Z}[x]$. So $L$ has algebraic integer coefficients.
\end{proof}

\begin{lem}
	If $f$ in $\overline{\Q}[x]$ has degree $>1$, then any polynomial $g \in \C[x]$ commuting with $f$ has coefficients in $\overline{\Q}[x]$.  \end{lem}
\begin{proof}
Since $f\in \overline\Q[x]$, the requirement that $f\circ g = g\circ f$ gives algebraic conditions on the coefficients of $g$. From Boyce~\cite{Boyce}, we know that there are finitely many $g$ of a fixed degree commuting with $f$, so this combined with the algebraic conditions on the coefficients ensures that $g\in \overline\Q[x]$. More specifically, the algebraic condition imposed by $f\circ g = g \circ f$ guarantees that the lead coefficient of $g$ is in $\overline\Q$, and Boyce describes an algorithm due to Jacobsthal~\cite{Jacobsthal}  for computing all of the coefficients of $g$ from this lead coefficient.
\end{proof}

Combining the above, we have the following result.
\begin{prop}\label{prop:coeff}
	If $f \in \overline{\Z}[x]$ is monic of degree $>1$ and is not conjugate to a power function or plus or minus a Chebyshev polynomial, then any polynomial $g$ commuting with $f$ has coefficients in $\overline{\Z}[x]$.
	\end{prop}

We now prove our desired theorem.

\begin{thm}
\label{e1-implies-e2}
	Let $f \in \Z[x]$ be a monic polynomial with integer coefficients.
	
	 Let $P \in \Z[x_1, \dotsc, x_n]$ be a primitive polynomial with integer coefficients.  Suppose that, working in the ring $\C[x_1, \dotsc, x_n]$,  $P$ divides some polynomial $Q \in \C[x_1, \dotsc, x_n]$ that is a product of factors of the form $\tilde{f}^n(x_i)  - L(\tilde{f}^m(x_j))$ where $m, n \ge 0$ are integers, $L \in \C[x]$ is a linear polynomial commuting with an iterate of $f$, and $\tilde{f} \in \C[x]$ is a non-linear polynomial of minimal degree commuting with an iterate of $f$ (with possibly different choices of $L$, $\tilde{f}$, $n$, and $m$ for each factor).

	 Then, working in the ring $\Z[x_1, \dotsc, x_n]$,  $P(x)$ divides some polynomial $R \in \Z[x_1, \dotsc x_n]$ that is a product of factors of the form $\tilde{f}^n(x_i)  - L(\tilde{f}^m(x_j))$ where $m, n, \tilde f$, and $L$ are as above.
\end{thm}

\begin{proof}

	Let $Q = Q_1,\dotsc, Q_n$ be all the Galois conjugates of $Q$.  Since the property of commuting with the polynomial $f$, which has $\Q$-coefficients, is preserved by the action of $\Gal(\overline{\Q}/\Q)$, the $Q_i$ all have factorizations of our desired form.
	
	Now let $R  = \prod_i Q_i$, which is also a product of factors of this form. Since all these factors have coefficients in $\overline{\Z}$, so does $R$. But also $R$ is invariant under the Galois action by construction, so $R \in \Q[x]$; therefore, $R \in \Z[x]$.
	
	Finally, we need to check that $P$ divides $R$ in the ring $\Z[x_1, \dotsc, x_n]$.  From the above, we know that, inside the ring $\C[x_1, \dotsc, x_n]$, $P$ divides $Q$, so also $R$. Since both $P$ and $R$ have rational coefficients, the quotient $R/P$ does also, and $P$ divides $R$ in $\Q[x_1, \dotsc, x_n]$.  Since additionally $P$ and $R$ have integer coefficients, and $P$ is primitive, by Gauss's Lemma $P$ divides $R$ in $\Z[x_1, \dotsc, x_n]$, as desired.
\end{proof}

\section{Dynamical Kronecker's Lemma in some two-variable cases}
\label{sec:ad}

In this section we prove that (a) implies (d) in Figure~\ref{fig:diagram} for polynomials $f$ that satisfy a certain condition. More precisely, we give an alternate proof of the following key step in the proof of two-variable Dynamical Kronecker, which replaces the assumption of the Dynamical Lehmer's Conjecture with the assumption that the preperiodic points of $f$ belong to its Julia set:

\begin{thm}\label{mahler-preperiodic}
Assume  that $\PrePer(f) \subseteq \J_f$.
	If $P(x ,y) \in \Z[x, y]$ with $\m_f(P) = 0$, then the graph of $P(x, y) = 0$ passes through infinitely many points $(\alpha, \beta)$ for which $\alpha$ and $\beta$ are both preperiodic for $P$.
\end{thm}

\begin{remk}The property $\PrePer(f) \subseteq \J_f$ is a strong assumption. In Section \ref{sec:J=K} we discuss some conditions that guarantee that this property is satisfied, and so Dynamical Kronecker's Lemma holds unconditionally for $f$.  
\end{remk}

Before proving Theorem \ref{mahler-preperiodic}, we consider two lemmas. 

\begin{lem}\label{leading-coefficient}
Assume  that $\PrePer(f) \subseteq \J_f$.
	If $P(x, y) = a(x)y^k + (\text{lower order terms in $y$}) \in \Z[x, y]$ is a two-variable polynomial with $\m_f(P) = 0$, then 
	\begin{enumerate}[(i)]
		\item $\m_f(a) = 0$;
		\item the polynomial $a(x) \in \Z[x]$ is primitive (the gcd of its coefficients is 1);
		\item $a(x)$ divides the polynomial $P(x, y)$ (in $\Z[x, y]$).
	\end{enumerate}
\end{lem}
\begin{proof}
	(i) This result follows from the equality case of Proposition~\ref{prop:convergence}. In particular, from
	equation~\eqref{eq:non-neg} we see that $\m_f(P)$ can be zero if and only if the two (non-negative)  summands are both zero, one of which corresponds to $\m_f(a)$ in this two-variable case.
		
	(ii) This follows from (i) and the fact that non-primitive polynomials have positive dynamical Mahler measure.
	
	(iii)
	It is enough to show that $a(x)$ divides $P(x, y)$ in $\C[x, y]$, since if this is the case, the polynomial $P(x, y)/a(x)$ must have rational coefficients, and by Gauss's Lemma, the coefficients must also be integers.

    We argue by contradiction, somewhat in the style of \cite[Lemma~3.20]{Everestward}.  Suppose not: then there is a  root $\alpha$ of $a(x)$ such that $P(\alpha, y)$ is not identically~0.  By the single-variable Dynamical Kronecker's Lemma (Lemma~\ref{lem:dynamicalKronecker}), since $\m_f(a) = 0$, the roots of $a(x)$ are in $\PrePer(f)$. By assumption, they are in $\J_f$.
      
	As in the proof of Proposition~\ref{prop:convergence}, we have a factorization
	\[
	P(x, y) = a(x) \prod_{i = 1}^k (y - g_i(x)),
	\]
	where the $g_i$ are algebraic functions which may have branch cuts or singularities.  In particular, plugging in the $\alpha$ above, we see that some $g_i(x)$ must have a pole at $x = \alpha$. We will show this leads to a contradiction.

	We can decompose the Mahler measure of $P$ as
	\[
	\m_f(P) = \m_f(a) + \sum_{i = 1}^k \int_{\J_f} p_\mu (g_i(x)) d\mu_f(x).
	\]
	
	Since all summands are non-negative, in order to have $\m_f(P) = 0$ we must have $\int_{\J_f} p_\mu(g_i(x)) d\mu_f(x)= 0$ for each $i$.
	
	On the other hand, if $g_i$ has a pole at $\alpha$ of order $r \in \Q$, we have a power series expansion
	\[
	g_i(x) = (x- \alpha)^{-r} + \dotsb 
	\]
	where the first term dominates near $x = \alpha$, so there must be some neighborhood $U$ of $\alpha$ such that $p_\mu(g_i(x)) > 1$ for $x \in U$. 
	
	Then, 
	
	\[
	0 = \int_{\J_f} p_\mu(g_i(x)) d\mu_f(x) \ge \int_{U \cap \J_f} p_\mu(g_i(x)) d\mu_f(x)\ge \int_{U \cap \J_f} d\mu_f(x)= \mu_f(U \cap \J_f) > 0,
	\]
	which gives the desired  contradiction. (The last step uses the fact that $\alpha$ lies in the support $\J_f$ of $\mu$.)
\end{proof}

After factoring out the leading term, we are reduced to considering the case when 
\[
P(x, y) = y^k + (\text{lower order terms in $y$})
\]
 as a monic polynomial in $y$.

\begin{lem}\label{monic-case}
	If $P(x, y) = y^k + (\text{lower order terms in $y$}) \in \Z[x, y]$ is a two-variable polynomial, monic in~$y$, with $\m_f(P) = 0$, then for any $\alpha \in \mathcal{J}_f$ the polynomial $P_{\alpha} \in \C[y]$ given by $P_{\alpha}(y) = P(\alpha, y)$ satisfies $\m_f(P_{\alpha}) = 0$.
\end{lem}

\begin{proof}
	The polynomial $P_{\alpha} \in \C[y]$ is monic, so  $\m_f(P_{\alpha}) \ge 0$ for all $\alpha$.
	
	However,
\begin{equation}\label{eq:intzero}
	0 = \m_f(P) = \int_{\mathcal{J}_f} \m_f(P_\alpha) d\mu_f(\alpha),
	\end{equation}
	from which we can immediately deduce that $\m_f(P_\alpha) = 0$ for almost all $\alpha\in \mathcal{J}_f$ (that is, except possibly on a set of invariant measure $0$).
	
 Next we prove that $\m_f(P_\alpha)$ is a continuous function of $\alpha$. To do this, write 
		$P(\alpha,y) = \prod_{i} (y-g_i(\alpha))\in \C[y]$, where as above the $g_i$ are algebraic functions.
By Jensen's formula,
\[
\m_f(P_\alpha) =  \sum_{i} p_{\mu_f}(g_i(\alpha)).
\]
By Proposition~\ref{potential-continuous}, $p_{\mu_f}$ is continuous, and therefore $\m_f(P_\alpha) $ is continuous as a function of $\alpha\in \C$.

Since the support of $\mu_f$ is exactly $\J_f$ (as discussed in the proof of~\cite[Theorem 6.5.8]{Ransford} or ~\cite[Theorem 2, page 169]{Steinmetz}), equation \eqref{eq:intzero} then implies that $\m_f(P_\alpha) = 0$ for every $\alpha \in \J_f$. 
\end{proof}

\begin{proof}[Proof of Theorem~\ref{mahler-preperiodic}]
	Write $P(x, y) = a(x)y^k + (\text{lower order terms in $y$}) \in \Z[x, y]$.
	
	Case 1: $a(x)$ is not constant.  By the single-variable Dynamical Kronecker's Lemma (Lemma~\ref{lem:dynamicalKronecker}), the roots of $a(x)$ are preperiodic for $f$, and by Part (iii) of Lemma~\ref{leading-coefficient}, $P(x, y)$ contains the vertical lines through those preperiodic $x$-coordinates.
	
	Case 2:  $a(x)$ is constant, so it equals $\pm 1$ by Part (i) 
	of Lemma~\ref{leading-coefficient}.  By flipping the sign as necessary, we may assume that $a(x) = 1$, so we are in the situation of Lemma~\ref{monic-case}.
	
	Now let $\alpha$ be any preperiodic point for $f$, and let its Galois conjugates be $\alpha= \alpha_1, \alpha_2, \dotsc, \alpha_n$.  Consider the polynomial $P_{\alpha_1}P_{\alpha_2} \dotsm P_{\alpha_n} \in \C[y]$: this has algebraic integer coefficients since $\alpha$ is an algebraic integer (here we are crucially using the fact that $f$ is monic) and $P \in \Z[x, y]$, and in fact rational integer coefficients since it is invariant under the Galois action.  Since  $\alpha_1, \dotsc, \alpha_n \in \PrePer(f)$, by our assumption they are also in $\J_f$, and by Lemma~\ref{monic-case}, $\m_f(P_{\alpha_1}P_{\alpha_2} \dotsm P_{\alpha_n}) = 0$.

	From the single-variable Dynamical Kronecker's Lemma, we conclude that all roots of $P_{\alpha_1}P_{\alpha_2} \dotsm P_{\alpha_n}$ are preperiodic for $f$.  Hence any point on the curve $P(x, y) = 0 $ with $x$-coordinate $\alpha = \alpha_1$ also has preperiodic $y$-coordinate. Since $P(x, y)$ is monic in $y$, the polynomial $P(\alpha, y)$ is nonzero for any such $\alpha$, and thus there is some $\beta$ for which $P(\alpha, \beta) = 0$. Since there are infinitely many choices for the periodic point $\alpha$ we get infinitely many points with both coordinates preperiodic.
\end{proof}

\section{Conditions for the preperiodic points of $f$ to lie in the Julia set $\J_f$} \label{sec:J=K}
Given the results of Section~\ref{sec:ad}, it is natural to ask how restrictive the hypothesis is that $\PrePer(f) \subseteq \J_f$. This seems to be a delicate question in general. 
In the case of unicritical polynomials --- polynomials with a unique critical point $\gamma \in \C$ --- we can answer the question completely.
We begin with some background on the connection between Julia sets, filled Julia sets, and periodic points. 

For  $f \in \C[z]$, a  fixed point $z_0$ is 
\begin{itemize}
\item
{\bf repelling} if $|f'(z_0)|>1$, 
\item
{\bf neutral} if $|f'(z_0)|=1$, and 
\item
{\bf attracting} if $|f'(z_0)|<1$. 
\end{itemize}
If $|f'(z_0)|>1$, then the  image under $f$ of a small neighborhood around $z_0$ expands so that $z_0$ ``repels'' nearby points. If $|f'(z_0)|<1$, the image under $f$ of small neighborhood around $z_0$ shrinks, so that $z_0$ ``attracts'' nearby points. The number $ f'(z_0)$ is called the {\bf multiplier} of the fixed point.
These ideas generalize to $n$-cycles by considering points on the cycle as fixed points of the iterated polynomial $f^n$. So to study an $n$-cycle containing the point $z_0$ and determine whether it is repelling, neutral, or attracting, we consider the absolute value of  
\[
\left.\frac{d f^n}{d z}\right|_{z=z_0} = \prod_{z_i \text{ on the cycle}} f'(z_i).
\]
 The equality comes from applying the chain rule to the derivative of $f^n$.

As noted in Section~\ref{sec:ArithDynIntro}, all preperiodic points of a polynomial $f$ lie in the filled Julia set $\K_f$. But in fact more is true: The Julia set $\J_f$ is the closure of the repelling periodic points of $f$~\cite[Theorem 6.9.2]{beardon2000iteration}. Since the Julia set $\J_f$ is completely invariant under $f$---meaning that $f(\J_f) = \J_f = f^{-1}(\J_f)$~\cite[Theorem 3.2.4]{beardon2000iteration}---we need only  
 determine when the nonrepelling 
cycles lie in the Julia set. 

Notice that it suffices to check this when $\J_f \subsetneq \K_f$, since in the case of $\J_f=\K_f$, the preperiodic points lie in $\J_f$ trivially. Observe that the condition  $\J_f=\K_f$ is guaranteed if $\J_f$ is totally disconnected. 

The following two results will be useful in the sequel.

\begin{thm}~\cite[Theorem 1.35 (a)]{Silverman-arithmetic-dynamical}\label{thm:silverman} Let $f(z) \in \C[z]$ be a polynomial of degree $d\geq 2$. 
Then $f$ has at most $d-1$ nonrepelling periodic cycles in $\C$. 
\end{thm}

\begin{thm}\cite[Corollary 8.2]{Sutherland} \label{thm:multiplierunity} If $z_0$ is a periodic point with multiplier $\left.\frac{d f^n}{d z}\right|_{z=z_0}=\lambda$ a root of unity, then $z_0\in \J_f$. 
\end{thm}

\subsection{The degree-2 case} We start by considering the case of $\deg(f)=2$. In this case, we can completely classify monic integer polynomials that fail to have all preperiodic points in the Julia set.

\begin{prop}\label{prop:preperjulia2} Let $f\in \Z[z]$ be monic and quadratic. Then $\PrePer(f) \subseteq\J_f$ unless $f$ is affine conjugate over $\Z$  to either $z^2$ or  $z^2-1$.
 \end{prop}
 
 To prove Proposition~\ref{prop:preperjulia2}, we need a couple of tools. First, note that any quadratic polynomial $f\in \C[z]$ is affine conjugate over $\C$ to a polynomial of the form $z^2+c$ with $c \in \C$. To see this, choose a conjugating function $L(z) = \frac z a- \gamma$ where $a$ is the leading coefficient of $f$ and $\gamma$ is the unique  critical point of $f$. Then $f^L$ will be monic and have a critical point at zero, which gives it the desired form.
  
Let $f_c(z)=z^2+c$. The Mandelbrot set is defined as 
\[\mathcal{M}_2=\left\{ c \in \C : \sup |f_{c}^n(0)|<\infty\right\}.\]

\begin{center}
\begin{figure}[h]
 \includegraphics[scale=.3]{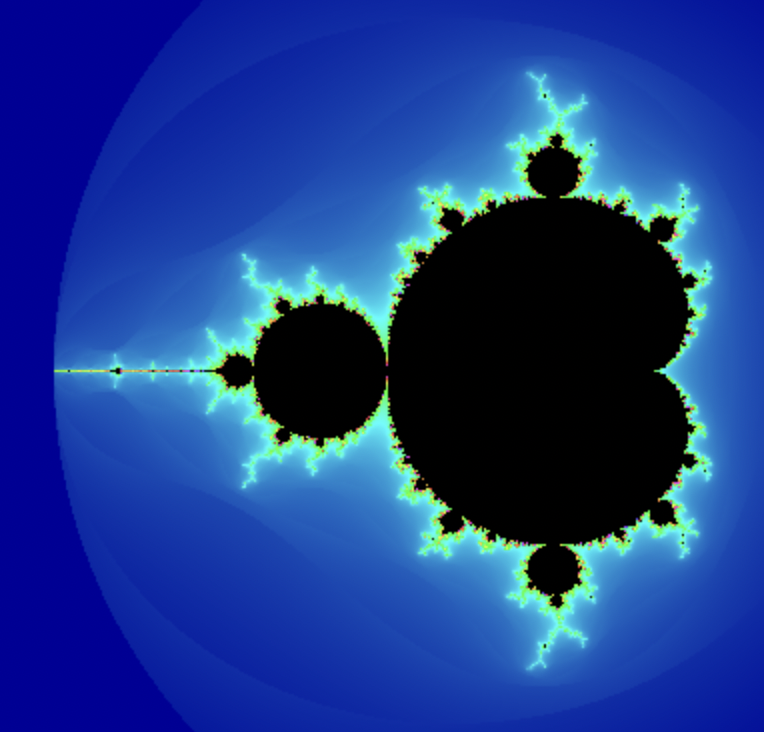} 
 \caption{The Mandelbrot set.}
 \label{Mandelbrot}
 \end{figure}
 \end{center}

The Mandelbrot set is contained in the disk of radius 2; furthermore,  $\mathcal{M}_2$ satisfies $\mathcal{M}_2\cap \R=\left[-2,\frac{1}{4}\right]$ (see~\cite[Chapter VIII, Theorem 1.2]{carleson2013complex}). 
It follows from the discussion of Julia sets in Section~\ref{sec:ArithDynIntro} that  for $c \not\in \mathcal M_2$, the Julia set for $f_c$ is totally disconnected and $\PrePer(f)\subseteq\J_f$.

The Mandelbrot set completely classifies conjugacy classes of complex quadratic  polynomials. In fact, the conjugacy described above  preserves the field of definition of a quadratic polynomial provided the field does not have characteristic~2. However, to prove Proposition \ref{prop:preperjulia2}, we need to understand conjugacy classes of monic integral quadratic polynomials, which is slightly more delicate.
\begin{lem}\label{lem:conj2}
Let $f\in\Z[z]$ be monic and quadratic. Then $f$ is affine conjugate over $\Z$  to an integral polynomial of the form $z^2 + c$ or $z^2 + z + c$.
\end{lem}

\begin{proof}
 Let $f(z)=z^2+\alpha z+\beta \in \Z[z]$. We will find a linear polynomial $L(z) \in \Z[z]$ such that  
  \[f = g^L=L^{-1}\circ g \circ L\]
  with $g$ a polynomial as in the statement. We have two cases:
  
  If $\alpha$ is even, take $\alpha_1 \in \Z$ such that $\alpha = 2\alpha_1$. Then for  $L=z+\alpha_1$  we have $g(z)=z^2+\beta-\alpha_1^2+\alpha_1$.
  
  If $\alpha$ is odd, take $\alpha_1 \in \Z$ such that $\alpha=2\alpha_1+1$. Then for $L=z+\alpha_1$  we have $g(z)=z^2+z+\beta-\alpha_1^2$.
\end{proof}

\begin{proof}[Proof of Proposition \ref{prop:preperjulia2}] By Theorem \ref{thm:silverman}, a degree-2 polynomial has at most one nonrepelling periodic cycle in $\C$. It suffices to find one such cycle in each case. 

First consider the case in which $f(z)$ is conjugate  to $g(z)=z^2+c$. Since  $\mathcal{M}_2\cap \R=\left[-2,\frac{1}{4}\right]$,
 it suffices to consider the cases of $c=0,-1,-2$.

When $c=0$, we immediately see that $z_0=0$ is an attracting fixed point in $\K_f\setminus \J_f$. \

When $c=-2$, we have $g = T_2$ (the second Chebyshev polynomial), and $\J_f=[-2,2]=\K_f$.

Finally, when $c=-1$, we find that $\{-1,0\}$ is an attracting cycle. Indeed $g^2(z)=z^4-2z^2$ and 
$\frac{d g^2}{d z}=4z^3-4z$. This gives 
\[
\left. \frac{d g^2}{d z}\right|_{z=0}=\left. \frac{d g^2}{d z}\right|_{z=-1}=0.
\] 
(This is a general phenomenon: When a critical point $f$ is strictly periodic, one can show that the multiplier of the cycle will be $0$, and the cycle is called {\bf superattracting}.)
 
Now we consider the case in which $f$ is conjugate  to $g(z)=z^2+z+c$. Letting $L(z) = z-\frac 1 2$, we find that  $g^L = z^2+c+\frac{1}{4} \in \Q[z]$. 
Again using the fact that $\mathcal{M}_2\cap \R=\left[-2,\frac{1}{4}\right]$, we have $\K_f=\J_f$ for $c<-\frac{9}{4}$ and $c>0$ and we only need to check $c=0,-1,-2$.

When $c=0$, we have $g(z)=z^2+z$, and $z_0=0$ is a neutral fixed point since $g'(z)=1+2z$ and $g'(0)=1$. 
Since the multiplier is $1$, Theorem \ref{thm:multiplierunity}  implies $0\in \J_f$.

When $c=-1$, we have $g(z)=z^2+z-1$ and $z_0=-1$ is a neutral fixed point since $g'(z)=1+2z$ and $f'(-1)=-1$. 
Since the multiplier is $-1$, Theorem \ref{thm:multiplierunity}  implies $-1\in \J_f$.
 
Finally, when $c=-2$, we have $g(z)=z^2+z-2$. We claim that the roots of $z^3 + 2z^2 - z - 1$ give a neutral 3-cycle with multiplier 1. First notice that 
\begin{equation}\label{eq:f3}
g^3(z)-z=
(z^2 - 2)(z^3 + 2z^2 - z - 1)^2.
\end{equation}
The roots of the first factor in \eqref{eq:f3} are (repelling) fixed points. Let $z_0$ be a root of $z^3 + 2z^2 - z - 1=0$. Since $z^3 + 2z^2 - z - 1$ has exponent 2 in the factorization of $g^3(z)-z$, it must be a factor of the derivative$\frac{d (g^3(z)-z)}{d z}$, but this implies that 
\[
\left. \frac{d (g^3(z)-z)}{d z}\right|_{z=z_0}=0,
\]
 from which we get $\left. \frac{d g^3}{d z}\right|_{z=z_0}=1$.
Since the multiplier is $1$, Theorem \ref{thm:multiplierunity}  implies $z_0\in \J_f$.
 \end{proof}

\subsection{The family $z^d+c$}

A polynomial of the form $f_{d,c}(z)=z^d+c$ has only one critical point, namely $z_0=0$. Therefore, in this family we have the dichotomy between connected and totally disconnected Julia sets,  and we can define the Mandelbrot or Multibrot set 
\[\mathcal{M}_d=\left\{ c \in \C : \sup |f_{d,c}^n(0)|<\infty\right\}.\]

\begin{figure}[ht!]
\captionsetup[subfigure]{justification=centering}
    \centering
    \begin{subfigure}[t]{0.3\textwidth}
        \centering
        \includegraphics[scale=.3]{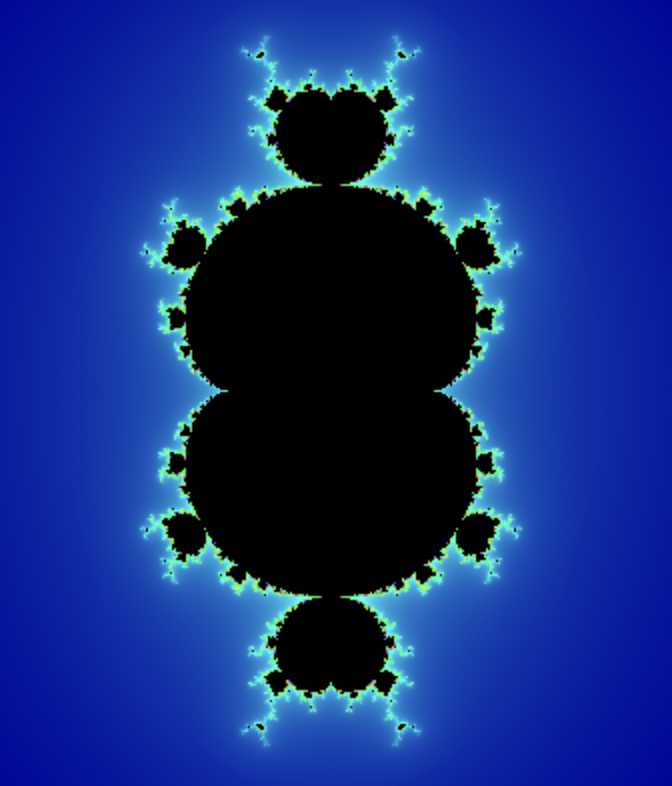}
        \caption{Multibrot set for $d=3$.}
    \end{subfigure}%
    ~ 
    \begin{subfigure}[t]{0.3\textwidth}
        \centering
        \includegraphics[height=1.2in]{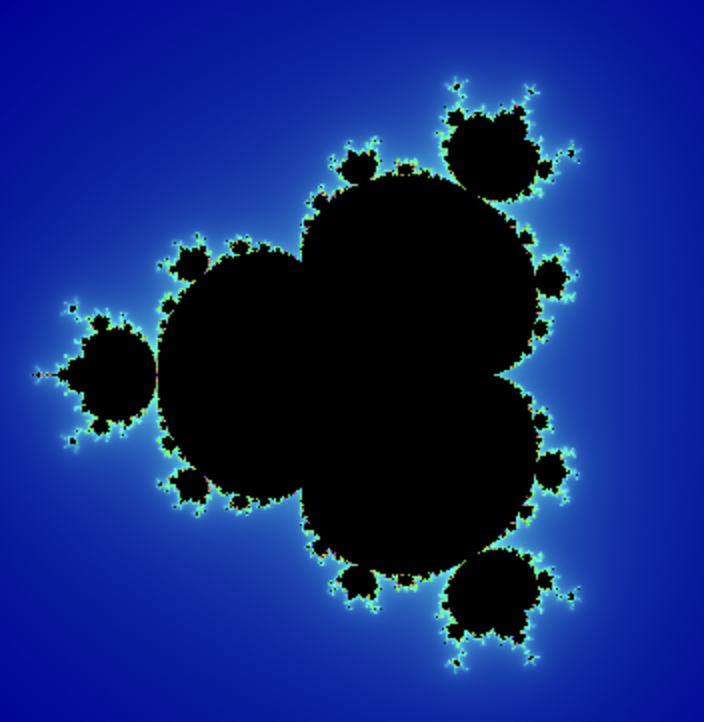}
        \caption{Multibrot set for $d=4$.}
    \end{subfigure}
    ~
        \begin{subfigure}[t]{0.3\textwidth}
        \centering
        \includegraphics[height=1.2in]{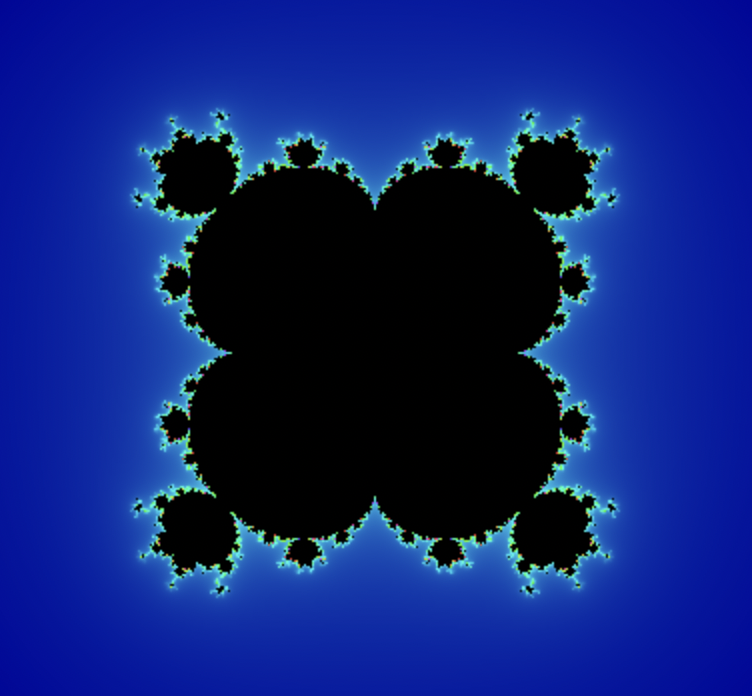}
        \caption{Multibrot set for $d=5$.}
    \end{subfigure}
\caption{Some Multibrot sets.}
\end{figure}

When $d$ is odd  and $d>1$, Paris\'e and Rochon~\cite{PariseRochon} proved 
\begin{equation}
\mathcal{M}_d\cap \R=\left [-\frac{d-1}{d^\frac{d}{d-1}},\frac{d-1}{d^\frac{d}{d-1}}\right],
\label{eq:dodd}
\end{equation}
and we remark that this implies $\mathcal{M}_d\cap \Z=\{0\}$ for $d>1$. 

When $d$ is even Paris\'e, Ransford, and Rochon~\cite{PariseRansfordRochon} proved 
\begin{equation}
\mathcal{M}_d\cap \R=\left[-2^\frac{1}{d-1},\frac{d-1}{d^\frac{d}{d-1}}\right ],
\label{eq:ceven}
\end{equation}
and we remark that this implies $\mathcal{M}_d\cap \Z=\{-1,0\}$ for $d>2$.

\begin{thm}\label{lem:conjcubic}
Let $f\in\Z[z]$ be a  monic polynomial that is affine conjugate over $\C$  to a polynomial of the form $z^d+c$ with $d>2$. Then $\PrePer(f) \not \subseteq\J_f$ if and only if either $c = 0$ or $d$ is even and $c = -1$.
\end{thm}
\begin{proof} 
If $f\in\Z[x]$ is  affine conjugate over $\C$ to $z^d+c$, then $f$ has a unique critical point $\gamma \in \C$. Since $f$ is monic,  the derivative factors over $\C$ as  $f'(z) = d(z-\gamma)^{d-1}$.  Integrating, we get $f(z) = (z-\gamma)^d + b$ where $\gamma, b \in \C$, but we also know that $f(z) \in \Z[z]$.
Of course, the derivative $f'(z) =d(z-\gamma)^{d-1} \in \Z[z]$ as well. The coefficient of $z^{d-2}$ in $f'(z)$ is $-d(d-1)\gamma$, so we see that $\gamma \in \Q$. We claim that when $d>2$, in fact $\gamma \in \Z$.

Looking at the coefficient of $z$ in $f(z)$, we have $d \gamma^{d-1} \in \Z$. Let $p$ be a prime dividing the denominator of $\gamma$. Since $d\gamma^{d-1} \in \Z$, we must have $p^{d-1} \mid d$. But this is impossible since $p^{d-1}\geq 2^{d-1}>d$ when $d>2$.
Now consider the constant term of $f$, which is $(-\gamma)^d +b$. Since $\gamma \in \Z$, we have $b\in \Z$ as well.

Choose $L(z) = z +\gamma$, and we see that $f^L(z) = z^d + b-\gamma \in \Z[z]$.  So $f$ is conjugate to a polynomial of the form $z^d+c$ with $c\in \Z$.  

We know that  $c \not \in \mathcal{M}_d$ implies  $\PrePer(f)\subseteq\J_f$.    From equations~\eqref{eq:dodd} and~\eqref{eq:ceven}, we deduce that for $c\in \Z$, $c\not \in \mathcal{M}_d$ iff  $c\not =0$ for $d$ odd and $c\not =0, -1$ for $d$ even. 
  Thus, it suffices to consider the exceptional cases $c=0$ and $c=-1$. 
 
 We already know that for the power function $z^d + 0$, the point $z_0=0$ is an attracting fixed point in $\K_f\setminus \J_f$. 
 
 Now consider $g(z)=z^d-1$ with $d$ even. In this case  we find that $\{-1,0\}$ is an attracting cycle. 
 Indeed $g^2(z)=(z^d-1)^d-1$ and $\frac{d g^2}{d z}=d^2z^{d-1}(z^d-1)^{d-1}$. This gives $\left. \frac{d g^2}{d z}\right|_{z=0}=\left. \frac{d g^2}{d z}\right|_{z=-1}=0$.
\end{proof}

One might hope that for $f \in \Z[z]$, if the coefficients of $f$ are sufficiently large, then all critical points will have unbounded orbit. In this case the Julia set would be totally disconnected, so that again $\PrePer(f) \subseteq \J_f$.

 \bibliographystyle{amsalpha}

\bibliography{Bibliography}

\end{document}